%% file: paper.tex
\documentclass[10pt]{amsart} 

\input{style.tex}

\input{macro.tex}

\title{Large smooth twins from short lattice vectors}

\author{Erik Mulder}
\address{University of Bergen, Bergen, Norway}
\curraddr{}
\email{erik.mulder@uib.no}
\thanks{}

\author{Bruno Sterner}
\address{Inria and École polytechnique, Institut Polytechnique de Paris, Palaiseau, France \& University of Waterloo, Waterloo, Canada}
\curraddr{}
\email{bsterner@uwaterloo.ca}
\thanks{}

\author{Wessel van Woerden}
\address{PQShield, Amsterdam, The Netherlands}
\curraddr{}
\email{wessel.vanwoerden@pqshield.com}
\thanks{}

\subjclass[2020]{Primary }

\date{}

\dedicatory{}

\begin{document}

\nolinenumbers

\begin{abstract}
    Finding the largest pair of consecutive $B$-smooth integers for a fixed $B$, also called a $B$-smooth twin, is computationally challenging. It has only been provably done for $B \leq 100$ and heuristically for $100 < B \leq 113$. 
    We improve this by detailing a new algorithm to find such smooth twins based on solving the shortest vector problem (SVP) in a well-constructed lattice.
    Using an estimate of the largest $B$-smooth twin that we develop in this work, we are able to significantly increase $B$ and notably report the heuristically largest twin with $B = 751$, which has $196$ bits.
    By slightly modifying the lattice, we are able to find even larger twins, but the resulting smoothness bound will not always be optimal.
    This notably includes a $213$-bit twin with $B = 997$, which is the largest twin found in this work.
\end{abstract}

\maketitle

\begin{displayquote}
	\begin{scriptsize}
	\emph{``Such pairs lead to sets of ``nearly" dependent logarithms of primes. For instance the number} $K = \log(63927525376) - \log(63927525375) = 13 \log(2) - 3 \log(3) - 3 \log(5) - 7 \log(7) + 4 \log(11) + \log(13) - \log(23) + \log(41)$,
	\emph{which cannot be zero because of the unique factorisation theorem, is, however, less than $1.56427 \cdot 10^{-11}$."}
	
	\begin{flushright}- D.H. Lehmer~\cite{lehmer1964problem} \end{flushright}
	\end{scriptsize}
\end{displayquote}

\input{1-introduction}
\subsection*{Acknowledgements.} The second author was supported by two projects: first the HYPERFORM\footnote{See \url{https://project.inria.fr/hyperform}.} consortium funded through Bpifrance and l’Agence nationale de la recherche through a Plan France 2030 grant (ANR-22-PETQ-0008 PQ-TLS); and second the NSERC Alliance Consortia Quantum Grant ``Accelerating the transition to quantum-resistant cryptography'' (ALLRP 578463–2022), and the NSERC Discovery Grant program\footnote{See \url{https://nserc-crsng.canada.ca/en/funding-opportunity/archived-alliance-consortia-quantum-grants-call-proposals-supporting}}.
Part of this work was executed while the third author was employed by IMB at Université de Bordeaux under the CHARM ANR-NSF grant (ANR-21-CE94-0003). Experiments presented in this paper were carried out using the PlaFRIM experimental testbed, supported by Inria, CNRS (LABRI and IMB), Université de Bordeaux, Bordeaux INP and Conseil Régional d’Aquitaine~\footnote{See \url{https://www.plafrim.fr}.}.
\input{2-smooth_prelims}
\input{4-heuristics}
\input{3-lattice_prelims}
\input{4-smooth_rat_lat}

\input{5-analysis}
\input{6-results}
\input{7-isogeny_impact}

\iffullversion
\else
\appendix
\input{supplementary_material}
\fi

\bibliography{bib}
\bibliographystyle{amsplain}

\iffullversion
\appendix
\input{supplementary_material}
\fi

\end{document}

%% file: style.tex
\newif\iffullversion
\fullversiontrue

\usepackage{amsmath,amssymb,amsfonts}
\usepackage{multirow}
\usepackage{arydshln}
\usepackage[mathscr]{eucal}
\usepackage[normalem]{ulem}
\usepackage{fancyhdr}
\usepackage{tikz}
\usepackage{url}
\usepackage{booktabs}
\usepackage[vlined,linesnumbered,ruled]{algorithm2e}
\usepackage[font=small,labelfont=bf,compatibility=false]{caption}
\usepackage{subcaption}
\usepackage{etoolbox}
\usepackage{booktabs}
\usepackage{rotating}
\usepackage{multirow}
\usepackage{enumitem}
\usepackage{pifont}
\usepackage{xcolor}
\usepackage{tikz-qtree}
\usepackage{pgfplots}
\pgfplotsset{width=7cm,compat=1.18} 
\definecolor{darkblue}{rgb}{0,0,0.5}
\definecolor{darkgreen}{rgb}{0,0.5,0}
\usepackage[colorlinks=true,linkcolor=darkblue,urlcolor=darkblue,citecolor=darkgreen]{hyperref}
\usepackage{mathtools}
\usepackage{cancel}
\usepackage{ifthen}
\usepackage{graphicx}
\usepackage[super]{nth}
\usepackage{bm}
\usepackage{csquotes}

\usepackage{color}
\usepackage[pagewise]{lineno}
\definecolor{light}{rgb}{0.8, 0.8, 0.8}

\newif\ifhidetodos
\hidetodosfalse

\ifhidetodos
\usepackage[final,multiuser,inline,nomargin,marginclue]{fixme}
\else
\usepackage[draft,multiuser,inline,nomargin,marginclue]{fixme}
\fi

%% file: macro.tex
\newtheorem{theorem}{Theorem}[section]
\newtheorem{lemma}[theorem]{Lemma}
\newtheorem{corollary}[theorem]{Corollary}
\newtheorem{proposition}[theorem]{Proposition}
\newtheorem{heuristic}[theorem]{Heuristic}
\newtheorem{conjecture}[theorem]{Conjecture}

\theoremstyle{definition}
\newtheorem{definition}[theorem]{Definition}

\newtheorem{remark}{Remark}

\newcommand{\Z}{\ensuremath{\mathbb{Z}}}
\newcommand{\N}{\ensuremath{\mathbb{N}}}
\newcommand{\Q}{\ensuremath{\mathbb{Q}}}
\newcommand{\R}{\ensuremath{\mathbb{R}}}
\newcommand{\F}{\ensuremath{\mathbb{F}}}

\newcommand{\bv}{{\bm{b}}}
\newcommand{\cv}{{\bm{c}}}

\newcommand{\vv}{{\bm{v}}}

\newcommand{\xv}{{\bm{x}}}
\newcommand{\yv}{{\bm{y}}}

\newcommand{\zero}{{\bm{0}}}

\newcommand{\Bm}{{\bm{B}}}
\newcommand{\Cm}{{\bm{C}}}

\newcommand{\lat}{\ensuremath{\mathcal{L}}}
\newcommand{\ball}{\ensuremath{\mathcal{B}}}
\newcommand{\gh}{\ensuremath{\operatorname{gh}}}
\newcommand{\rk}{\ensuremath{\operatorname{rk}}}
\newcommand{\vol}{\ensuremath{\operatorname{vol}}}
\newcommand{\spa}{\ensuremath{\operatorname{span}}}
\newcommand{\val}{\ensuremath{\operatorname{val}}}
\newcommand{\norm}[1]{\lVert #1 \rVert}

\newcommand*\pnx[1]{\ensuremath{2x^{#1}-1}}

\newcommand*\fpnx[1]{\ensuremath{p_{#1}(x) = \pnx{#1}}}

\newenvironment{proof*}[1]
  {\renewcommand{\proofname}{#1}%
   \begin{proof}}
  {\end{proof}}

%% file: 1-introduction.tex
\section{Introduction}\label{sec:intro}


In 1991, during what can be considered the golden age of integer factorisation, Schnorr~\cite{schnorr1991factoring} proposed a factoring method\footnote{Which was independently analysed by Adleman~\cite{adleman1995factoring}} based on solving the closest vector problem in the so-called \emph{prime number lattice}.
While his strategy looked promising, the growing and necessary lattice dimension when you increase the input integer provides an obstruction to the feasibility of the algorithm. As a result, it did not compete with other factoring algorithms such as the number field sieve~\cite{lenstra1993development}.

Many modern integer factoring techniques use smooth integers which are integers with only small prime factors. We quantify this properly to say that an integer $n$ is $B$-smooth if $p \mid n$ $\Rightarrow$ $p \leq B$ for all primes $p$, where we refer to $B$ as the \emph{smoothness bound}. We refer to\iffullversion~Granville's survey \fi~\cite{granville2008smooth} for the factoring relevance of these integers and beyond.

Recently the problem of finding pairs of consecutive smooth integers, which we call smooth twins, has attracted attention. The primary relevance of these smooth twins comes from isogeny-based cryptography. Certain cryptosystems, including the original versions of SQIsign~\cite{SQISign,luca2022new} and a variant of POK{\'E}~\cite{basso2025poke,santos2024finding}, either work with large smooth twins whose sum is prime or a small variant of this. \iffullversion \else 
To formalise this, we call a pair of consecutive integers $(r,r+1)$ a \emph{$B$-smooth twin} if their product $r \cdot (r+1)$ is $B$-smooth and drop $B$ from this notation when it is implicit from the context. We call a $B$-smooth twin \emph{strictly} $B$-smooth if it is not $(B-1)$-smooth. \fi 

\iffullversion
We formalise the definition of $B$-smooth twins.

\begin{definition}\label{def:smooth}
    We call a pair of consecutive integers $(r,r+1)$ a \emph{$B$-smooth twin} if their product, $r \cdot (r+1)$, is $B$-smooth and drop $B$ from this definition when it is implicit from the context.
    When $B$ is included we call it the \emph{smoothness bound}. We call a $B$-smooth twin \emph{strictly} $B$-smooth if it is not $(B-1)$-smooth.
\end{definition}
\fi

It is known that the set of $B$-smooth twins is finite~\cite{Stormer} and one can find all such twins by solving exponentially many Pell equations in $B$. This has been done for $B = 100$~\cite{luca2011largest} and heuristically for $B = 113$~\cite{BSIDH} -- the largest of these twins have $64$ and $75$ bits (respectively). \iffullversion We give terminology for the largest $B$-smooth twin. \else We call the largest $B$-smooth twin an \emph{optimal twin}. In the other direction, for an input bit size $b$, the \emph{optimal smoothness bound} is the smallest smoothness bound $B$ such that there is a $b$-bit $B$-smooth twin. \fi  

\iffullversion
\begin{definition}\label{def:optimal}
    We call the largest $B$-smooth twin an \emph{optimal} twin. In the other direction, for an input bit size $b$, the \emph{optimal smoothness bound} is the smallest smoothness bound $B$ such that there is a $B$-smooth twin $(r,r+1)$ with $r > 2^{b-1}$.
\end{definition}
\fi

For cryptographic relevance one requires $256$-bit smooth twins. We call these \emph{cryptographic smooth twins}. The optimal smoothness bound for such a twin is $B \approx 1250$. So finding them from solving (at least a large proportion of) Pell equations is far beyond computational reach. Other methods~\cite{BSIDH,costello2021sieving,sterner2025towards} exist that find such cryptographic smooth twins but one obtains larger smoothness bounds. 


\subsection*{Contributions.} In this work we revisit Schnorr's factoring idea and repurpose it to find large smooth twins. We observe that the shortest vectors in the prime number lattice often correspond to smooth twins. So using state-of-the-art SVP algorithms~\cite{albrecht2019general,ducas2021advanced,zhao2025sieving} we find short vectors in the lattice and see whether they correspond to smooth twins. By parametrising the lattice correctly, one can search for smooth twins of a particular size including the optimal twins. With current SVP algorithms, we can find heuristically optimal twins with much larger smoothness bounds. For instance, we found the following $196$-bit $751$-smooth twin 
\begin{align}
\begin{split} \label{eqn:opt751}
	r & = 7^7 \cdot 11 \cdot 17 \cdot 29 \cdot 47 \cdot 59 \cdot 67 \cdot 83^2 \cdot 89 \cdot 151^3 \cdot 163 \cdot 173 \cdot 271 \cdot 347 \cdot 461  \\
	& \qquad \cdot 491 \cdot 547\cdot 587 \cdot 619 \cdot 661  \cdot 683 \cdot 701, \mbox{ and} \\
	r+1 & = 2 \cdot 3^9 \cdot 13^2 \cdot 19 \cdot 31 \cdot 41 \cdot 71 \cdot 73 \cdot 97 \cdot 157^2 \cdot 181^3 \cdot 191 \cdot 227 \cdot 241 \cdot 293 \\
	& \qquad \cdot 307^3 \cdot 337\cdot 557 \cdot 617 \cdot 727 \cdot 751
\end{split}
\end{align}
which we believe is optimal. As mentioned earlier, the only known provable optimal twins have $B \leq 100$. For $100 < B \leq 113$ optimal twins are only known heuristically. 

In order to convince the reader that this $751$-smooth twin is likely the optimal twin, we develop theoretical bounds and estimates for optimal twins based on some heuristics. In particular, the standout result is the following. 
\begin{theorem} \label{thm:opt}
	Assuming Heuristic~\ref{heu:opt}, as $B \to \infty$ we have the following asymptotic for optimal $B$-smooth twins $(r,r+1)$:
	\begin{align*}
		r \sim e^{2e\sqrt{B}}.
	\end{align*}
\end{theorem}
This result explains the asymptotic behaviour of optimal twins, but for concrete and small $B$ this estimate deviates from the precise size. For this regime one can use numerical techniques to get a better approximation for optimal twins.

Already in 1975, Erd\H{o}s essentially showed, under analogous heuristic assumptions, that the largest $B$-smooth twin grows as $r \sim e^{O(\sqrt{B})}$ \cite{erdos1975problems}.
We improve on this by making the constant in the exponent explicit.

By introducing some trade-offs in the lattice, one can find even larger twins. One cannot expect the resulting twins to be optimal using these trade-offs. For instance we found the following $213$-bit $997$-smooth twin
\begin{align}
\begin{split} \label{eqn:twin997}
	r & = 19^2 \cdot 41 \cdot 43^2 \cdot 53 \cdot 59^2 \cdot 73^2 \cdot 83 \cdot 173 \cdot 227 \cdot 241 \cdot 281 \cdot 337 \cdot 397^2 \cdot 433 \\
	& \qquad \cdot 541 \cdot 577 \cdot 593 \cdot 787 \cdot 821 \cdot 839 \cdot 857^2 \cdot 967, \mbox{ and} \\
		r+1 & = 2^2 \cdot 3 \cdot 5^2 \cdot 13^3 \cdot 23 \cdot 37 \cdot 47 \cdot 79 \cdot 107 \cdot 127 \cdot 131 \cdot 151 \cdot 157 \cdot 167 \cdot 179 \\
		& \qquad \cdot 181^2 \cdot 193 \cdot 223 \cdot 283 \cdot 317 \cdot 367 \cdot 379 \cdot 601 \cdot 709^2 \cdot 743 \cdot 941 \cdot 997
\end{split}
\end{align}
and, with the estimates that will be developed in Section~\ref{sec:smoothheuristic}, we suspect that the optimal $997$-smooth twin should have approximately $227$ bits. We discuss the challenges to find larger smooth twins relevant for cryptography.

For some fixed values of $B$, we attempt to find all $B$-smooth twins from our approach by expanding a known set of smooth twins~\cite{bruno2022cryptographic}. In particular we conjecture to have the complete set of $200$-smooth twins. The only known way to verify this is to comprehensively solve $2^{46}$ Pell equations which is computationally challenging. 


\iffullversion
In addition, we put our algorithm in a cryptographic context. While we cannot find cryptographic smooth twins, we are still able to use smaller twins and boost these to cryptographic parameters for the original version of the SQIsign signature scheme~\cite{SQISign}. These parameters have smaller smoothness bounds than those submitted to \emph{NIST's call for additional post-quantum signatures}~\cite{chavez2023sqisign}. With a theoretical decrease of $19.85\%$ on the signing time, these new parameters are a competitive alternative that could be considered for future use.
\else
In addition, we put our algorithm in a cryptographic context and get new parameters for the original version of SQIsign~\cite{SQISign} with smaller smoothness bounds. With a theoretical decrease of $19.85\%$ on the signing time, these new parameters are a competitive alternative that could be considered for future use.
\fi



A Python3 implementation of our algorithm is publicly available at
\begin{center}
	\url{http://github.com/WvanWoerden/SmoothTwins},
\end{center}
and a dataset containing the conjectured complete set of $200$-smooth twins, and some of the largest smooth twins for $B < 1100$, is publicly available at~\cite{zenodo}
\begin{center}
	\url{https://zenodo.org/records/18234089}.
\end{center} 

\subsubsection*{Notation.} For a smoothness bound $B$, $P_B \coloneqq \{2,3,\dots,q\} = \{ p \leq B \}$ will be the set of primes up to $B$ with cardinality $\pi(B) = \# P_B$. We will denote $\log( \cdot )$ and $\log_2( \cdot )$ for the natural and base-$2$ logarithm respectively. We denote $\val_p(y)$ for the valuation of $y \in \Q$ at $p$. For two functions $f,g : \N \to \R$ we write $f(n) = O(g(n))$ if there is some $C > 0$ and $N \in \N$ such that $|f(n)| \leq C |g(n)|$ for all $n \geq N$. We write $f(n) = \Omega(g(n))$ if $g(n) = O(f(n))$; $f(n) = o(g(n))$ if $f(n)/g(n) \to 0$ as $n \to \infty$; and also $f(n) \sim g(n)$ if $f(n)/g(n) \to 1$ as $n \to \infty$. 

%% file: 2-smooth_prelims.tex
\section{Existing Methods for Finding Smooth Twins}\label{sec:prior}
\iffullversion
We give a short summary of the algorithms that exist to find smooth twins.

\subsection{Constructive Methods} \label{subsec:consmethod}

The methods presented here attempt to find all or almost all $B$-smooth twins for a fixed and small smoothness bound $B$. 

\else

We give a short summary of two methods that find all or almost all $B$-smooth twins for a fixed and small smoothness bound $B$.

\fi

\subsubsection*{Solving Pell equations.} Let $x=2r+1$ so that $x^2-1 = 4r(r+1)$ is $B$-smooth. Decompose this into its squarefree part, $D$, and its square part, $y^2$. Then the pair $(x,y)$ is a solution to the Pell equation $X^2 - DY^2 = 1$. One reverses this to solve all $2^{\pi(B)}$ Pell equations, one for each $D = \prod_{\ell \in P_B} \ell^{e_{\ell}}$ with $e_{\ell} \in \{ 0,1 \}$. Lehmer~\cite{lehmer1964problem} gives provable conditions on these equations to ensure all $B$-smooth twins are found.


This is computationally infeasible for large $B$ -- the complexity to solve these equations is at least $2^{\pi(B) + o(\pi(B))}$. This ignores the fact that the fundamental solution is (in the worst case) doubly exponential in $B$. Finding such solutions would add another exponential to this complexity\footnote{We refer to~\cite[\S6]{lehmer1964problem} and~\cite{luca2011largest} for some discussion on getting around this.}. In practice one can adopt an early abort strategy that only considers solutions up to some reasonable bound.
If this bound is exponential in $B$, then processing each discriminant becomes polynomial in $B$ and the total complexity becomes $2^{\pi(B) + o(\pi(B))}$. With this, one could heuristically get all $B$-smooth twins (as done for $B = 113$~\cite{BSIDH}).

%

\iffullversion
\begin{remark}
	The authors of~\cite{buzek2022finding} found cryptographic smooth twins in this manner. Instead of solving all Pell equations they only solved equations with a small number of prime factors in $D$. This requires choosing a much larger $B$ than the optimal.
\end{remark}
\fi


\subsubsection*{Conrey-Holmstrom-McLaughlin (CHM)} This algorithm~\cite{conrey2013smooth} recursively builds an increasing chain of $B$-smooth twins $S^{(0)} = \{ 1 , \dots , B-1 \} \subseteq S^{(1)} \subseteq \dots \subseteq S^{(i)} \subseteq S^{(i+1)} \subseteq \dots \subseteq S^{(n-1)} = S^{(n)}$ until no more twins are found (which is ensured by St{\o}rmer's~\cite{Stormer} result). For each $r,s \in S^{(i)}$ with $r < s$ compute 
\begin{align*}
	\frac{t}{t^{\prime}} = \frac{r}{r+1} \cdot \frac{s+1}{s},
\end{align*}
where $\gcd(t,t^{\prime}) = 1$, and include $t$ in the next set $S^{(i+1)}$ when $t^{\prime} = t+1$. This produces a $B$-smooth twin $(t,t+1)$ and shows the correctness of the recursion.

This procedure produces a large but incomplete set of $B$-smooth twins. It is conjectured that the proportion of $B$-smooth twins not found with CHM is $o(1)$ as $B \to \infty$. The experiments in~\cite{bruno2022cryptographic} report the results with $B = 547$. The largest twin found in that run has $122$ bits which we will see later is far from optimal. 

\iffullversion
\subsection{Probabilistic Methods}  \label{subsec:probmethod}

The methods described here search for smooth twins of a fixed target size including cryptographic smooth twins. We give a brief account here and refer to~\cite{costello2021sieving,sterner2025towards} for more concrete details.

 
\subsubsection*{Searching using polynomial pairs.} The state-of-the-art here is to work with polynomial pairs $f,g \in \Z[x]$ with $g - f \equiv C \in \Z_{>0}$. One carefully evaluates these polynomials to get a smooth twin: $(r,r+1) = (f(m)/C,g(m)/C)$. This strategy is most effective when $\deg(f)$ is not too large and $f,g$ admit nice factorisations. This includes the circumstance when $f,g$ are a product of linear factors~\cite{costello2021sieving}; as well as those with some repeated and slightly larger degree factors~\cite{BSIDH,sterner2025towards}.


For smooth twins in the bit range $(240,256]$ the smallest $B$ reported with these polynomials is $B = 17341$. For those with a prime sum, the smallest is $B = 32039$. 

\fi

%% file: 4-heuristics.tex
\section{Smoothness Heuristics}\label{sec:smoothheuristic}

%


The Dickman-De Bruijn~\cite{dickman1930frequency,de1951asymptotic} function is the continuously differentiable function $\rho : [0,\infty) \to [0,1]$ that satisfies the following difference differential equation
\begin{align}
\begin{split} \label{eqn:dickdefn}
	\rho(u) & = 1, \qquad \qquad \qquad \quad (0 \leq u \leq 1); \\
	u\rho^{\prime}(u) & = -\rho(u-1), \qquad \qquad \quad (u > 1).
\end{split}
\end{align}
Other than $\rho(u) = 1 - \log(u)$ for $1 < u \leq 2$, there is no closed form for $\rho(u)$ in terms of elementary functions for arbitrary $u > 1$. A simple approximation of this function is $\rho(u) \approx u^{-u}$ and a better one is given by $\rho(u) \approx e^u (u\log(u))^{-u}$. These approximations are derived from the asymptotic expression~\cite[\S3.9]{granville2008smooth}
\begin{align}
    \label{rho_approx}
	\rho(u) = \left( \frac{e+o(1)}{ u\log(u)} \right)^{u},
\end{align}
as well as a more precise statement due to De Bruijn~\cite[Eqn~(1.8)]{de1951asymptotic}. From a computational perspective, numerical methods~\cite{van1969numerical} exist for evaluating this function which are built into many popular computer algebra packages. 

This function is related to the distribution of smooth integers at an asymptotic level. More precisely, if $\Psi(N,B)$ denotes the number of $B$-smooth integers $m \leq N$, then for a fixed $u > 0$ and as $N \to \infty$ we have~\cite{dickman1930frequency,de1951asymptotic}
\begin{align*}
	\Psi(N,N^{1/u}) \sim \rho(u)N.
\end{align*}
While this expression is asymptotic, experimentally $\rho(u)$ is a good estimate for the probability that an integer of size $N$ is $N^{1/u}$-smooth.

\iffullversion
Throughout this article we are interested in smooth integers with small smoothness bounds. In this scenario, for $A > 1$ and as $N \to \infty$, it is known that 
\begin{align*}
	\Psi(N,\log(N)^A) = N^{1-1/A+o(1)}.
\end{align*}
This is a special case of a result due to Rankin which is detailed in~\cite[\S3.10]{granville2008smooth}.
\fi

\subsection{Heuristics on the optimal twin} \label{sub:optimalestimate}

It is believed that the smoothness property in very short intervals is mutually independent.
A more formal and general statement is conjectured which has only been proven in a small range~\cite{Martin1999AnAF}. This fact has been utilised in practice~\cite{costello2021sieving,sterner2025towards} owing to its accuracy for concrete instances.
This conjecture implies that the probability that two close integers of size $N$ are $N^{1/u}$-smooth is $\rho(u)^2$.
Therefore, for large enough $N$, the number of $N^{1/u}$-smooth twins in the interval $[\frac{N}{2}, \frac{3N}{2}]$ is heuristically roughly $\rho(u)^2 N$.
This means that if $\rho(u)^2 N = 1$, then we expect roughly one $N^{1/u}$-smooth twin of size $N$, and not many greater than $N$.
This gives us the following heuristic on the optimal $B$-smooth twin\iffullversion , which we generalise to more than two consecutive smooth numbers. \else . We will use this heuristic in combination with~\eqref{rho_approx} to prove Theorem~\ref{thm:opt}. \fi


\iffullversion
\begin{heuristic} \label{heu:opt}
    Fix an integer $m \geq 2$ and a sequence of small positive integers $0 = a_0 < a_1 < \dots < a_{m-1} < m^{1+O(1)}$. As $B \to \infty$, the largest integer $r$ such that $r+a_i$ are all $B$-smooth satisfies
	\begin{align*}
		\rho(u) \sim r^{-1/m} 
	\end{align*}
	where $u = \log(r)/\log(B)$. For the special case giving $B$-smooth twins ($m = 2$ and $a_1 = 1$) the above condition is $\rho(u) = 1/\sqrt{r}$.
\end{heuristic}
\else
\begin{heuristic} \label{heu:opt}
	As $B \to \infty$, the largest $B$-smooth twin $(r,r+1)$ satisfies
	\begin{align*}
		\rho(u) \sim \frac{1}{\sqrt{r}}
	\end{align*}
	where $u = \log(r)/\log(B)$.
\end{heuristic}
\fi

\iffullversion
We get some validity of Theorem~\ref{thm:opt} assuming this heuristic when using the above estimate for $\Psi(N,\log(N)^A)$ with $A = 2$ and $N = r$. To make this more precise we need to solve the equation in Heuristic~\ref{heu:opt}. For instance, solving $\rho(u) = 1/\sqrt{r}$ using the crude estimate $\rho(u) \approx u^{-u}$ gives $\log(r) \approx \log(B) \sqrt{B}$. This additional factor $\log(B)$, which deviates from our Theorem~\ref{thm:opt}, emerges since this approximation deviates quite a bit from its actual value for large $u$. We improve this by replacing the crude estimate with the asymptotic expression in ~\eqref{rho_approx}.
\fi

As a precursor, we need the principal branch of the Lambert $W$ function -- the function $W_0 : \R_{>0} \to \R_{>0}$ such that for any $x > 0$ and $y = W_0(x) > 0$ we have $ye^y = x$. It is known that $W_0(x) = \log(x) - \log(\log(x)) + o(1)$ as $x \to \infty$~\cite{corless1996lambert}. We need a slight modification of this which is summarised in this short lemma.
%
\begin{lemma} \label{lem:invlog}
	Consider the function $x \mapsto x \log(x)$ for $x > 1$ (so that $x \log(x) > 0$). Then the principal branch of its inverse is the function $x \mapsto x/W_0(x)$ for $x > 0$.
\end{lemma}

\begin{proof}
	Starting from $y \log(y) = x$, with $x > 0$ and $y > 1$, we express $y$ solely in terms of $x$ by first rewriting this as $(x/y)e^{x/y} = x$ and then obtain $x/y = W_0(x)$.
\end{proof}


\begin{proof*}{Proof of Theorem~\ref{thm:opt}}
	We start with $\rho(u) = 1/\sqrt{r}$ from Heuristic~\ref{heu:opt}. Using the asymptotic expression from~\eqref{rho_approx}
	and $u = \log(r)/\log(B)$, this can be rearranged to
	\begin{align}
		u\log(u) = (e+o(1)) \sqrt{B}. \label{eqn:optproof}
	\end{align}	
	We need to address the $o(1)$ term here since this is a priori with respect to $u$ and not $B$.
    Write $g(u)$ for this $o(1)$ term. We will show that $\lim_{B \to \infty} u = \infty$, which implies that $\lim_{B \to \infty} g(u) = 0$ as well. Let $T \coloneqq \liminf_{B \to \infty} |e + g(u)|$.
	
    Suppose that $T = 0$.
    Note that the optimal $B$-smooth twin satisfies $r \geq B^2 - 1$ and hence $u = \log(r)/\log(B) \geq 1.5$ for large enough $B$.
    So the denominator and exponent of~\eqref{rho_approx} are positive for large enough $B$ and we get $\liminf_{B \to \infty} \rho(u) = 0$.
    Since $\rho(u)$ is a strictly descending function for $u > 1$ with limit 0 we must have $\lim_{B \to \infty} u = \infty$.
    This contradicts the assumption that $T = 0$.

    Hence we know that $T > 0$.
    By ~\eqref{eqn:optproof} we have $\liminf_{B \to \infty} |u \log(u)| = \infty$ and, as the map $x \mapsto |x\log(x)|$ has no poles other than $\infty$, we must have $\lim_{B \to \infty} u = \infty$.
    This implies that the $o(1)$ in~\eqref{eqn:optproof} is also with respect to $B$. Using Lemma~\ref{lem:invlog} we can express $u$ as
	\begin{align} \label{eqn:mainproof}
		u & = \frac{(e+o(1)) \sqrt{B}}{W_0\bigl((e+o(1)) \sqrt{B}\bigr)} = \frac{(e+o(1)) \sqrt{B}}{1 + \frac{1}{2}\log(B) - \log\bigl(1+\frac{1}{2}\log(B)\bigr) + o(1)}.
	\end{align}
    We can now conclude that
    \[ \lim_{B \to \infty} \log(r)/\sqrt{B} = \lim_{B \to \infty} u \log(B)/\sqrt{B} = 2e. \qedhere \]
\end{proof*}



\iffullversion
One can mimic this proof of Theorem~\ref{thm:opt} in the more general setting of smooth integers in short intervals. We simply state the result in the following theorem. 

\begin{theorem}[More generally]
	Fix an integer $m \geq 2$ and a sequence of small positive integers $a_0 = 0 < a_1 < \dots < a_{m-1} < m^{1+O(1)}$. Let $r$ be the largest integer such that $r+a_i$ is $B$-smooth for all $i = 0,\dots,m-1$. Assuming the generalised form of Heuristic~\ref{heu:opt}, as $B \to \infty$ we have the asymptotic
	\begin{align*}
		r \sim e^{meB^{1/m}}.
	\end{align*}
\end{theorem}
\fi

\subsubsection*{Numerical computations for small $B$.}

Theorem~\ref{thm:opt} gives a strong description of the asymptotic behaviour of optimal $B$-smooth twins. However for concrete and small $B$ this is not effective and one needs a more explicit solution to the equation $r \rho(u)^2 = 1$ underlying Heuristic~\ref{heu:opt}. One cannot solve this equation exactly and so we resort to numerical techniques to approximate the solution. 

One starts with an initial guess for $r$, such as the one given in Theorem~\ref{thm:opt}, and uses a root-finding algorithm, such as the \emph{bisection} or \emph{Newton's method}, to iteratively get a better approximation for $r$.
One can exploit the differentiable condition (see~\eqref{eqn:dickdefn}) for Newton's method and get a faster convergence rate. 
But we found that the bisection method was sufficiently fast for the values of $B$ we were interested in.

\iffullversion
\begin{figure}[ht]
	\centering
	\resizebox{\textwidth}{!}{
	\includegraphics{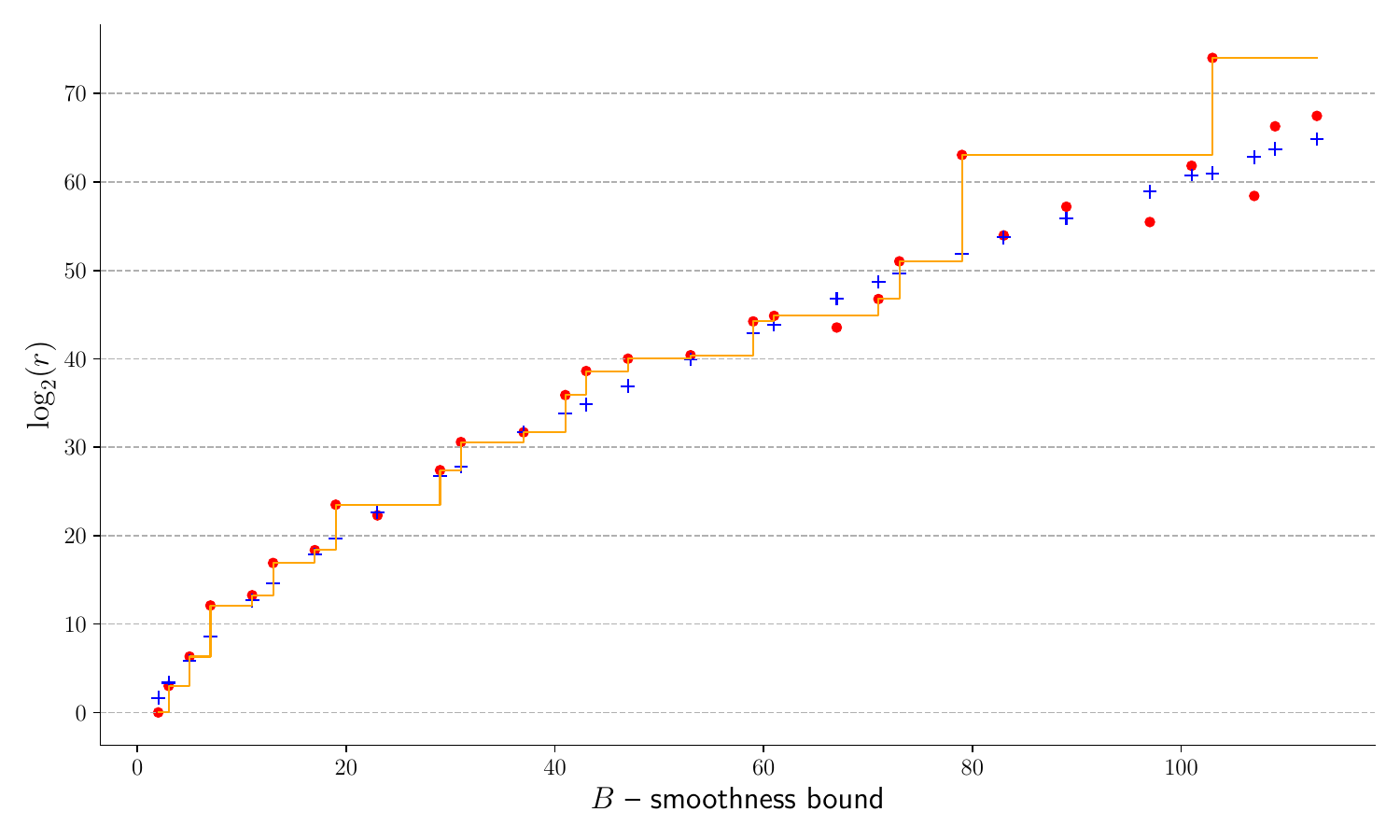}}
	\caption{Comparing the numerical estimates for optimal $B$-smooth twins (marked with a blue plus) against the provable and heuristic optimal twins with $B \leq 113$. The red dots mark the sizes of the largest \emph{strictly} $B$-smooth twins, while the orange line indicates the largest twin up to smoothness bound $B$. }
    
	\label{fig:opttwins113}
\end{figure}
\fi

For small $B \leq 113$ these numerical estimates for optimal twins correlate well to actual optimal twins found by solving Pell equations, see \iffullversion Figure~\ref{fig:opttwins113} \else Figure~\ref{fig:opttwins547}\fi. As a result we estimate the size of optimal twins for larger $B$, see Table~\ref{tab:optsize}. 

\begin{table}
	\centering
	\renewcommand{\tabcolsep}{0.2cm}
	\renewcommand{\arraystretch}{1.15}
	\resizebox{0.875\textwidth}{!}{
	\begin{tabular}{cc}
		Smoothness bound $B$ & Approximate $\log_2(r)$ \\
		\hline
		$200$ & $91.138890$ \\
		$500$ & $154.417132$ \\
		$750$ & $193.534196$ \\
		$1000$ & $226.657489$ \\
		$1100$ & $238.750073$ \\
		$1250$ & $255.920233$
	\end{tabular}\hspace{1em}
	\begin{tabular}{cc}
		Smoothness bound $B$ & Approximate $\log_2(r)$ \\
		\hline
		$1500$ & $282.426575$ \\
		$2000$ & $329.581941$ \\
		$2500$ & $371.202628$ \\
		$3000$ & $408.877910$ \\
		$4500$ & $506.210979$ \\
		$10000$ & $767.101783$
	\end{tabular}}
	\vspace{0.5ex}
	\caption{Heuristic size of optimal twins for various smoothness bounds $B$, i.e. approximate solutions to $\rho(u) = 1/\sqrt{r}$.}
	\label{tab:optsize}
\end{table} 


%% file: 3-lattice_prelims.tex
\section{Lattice Theory}\label{sec:latticeprelims}
A lattice $\lat$ is a discrete subgroup of the Euclidean space $\R^m$. Equivalently, any lattice $\lat$ can be described by a basis $\Bm \in \R^{m \times n}$ such that
$$\lat = \lat(\Bm) := \left\{ \sum_{i=1}^n x_i \bv_i : x_i \in \Z  \right\} ,$$
with $\R$-linearly independent columns $\bv_1, \ldots, \bv_n$. We denote $\rk(\lat) = n$ for the \emph{rank} of the lattice, which coincides with the dimension of its $\R$-linear span.

The discrete nature of a lattice allows us to define some geometric properties.~The minimal distance between any two distinct lattice points is called its \emph{first minimum}, denoted by $\lambda_1(\lat)$. Equivalently, it equals the length of the shortest non-zero vector:
$$\lambda_1(\lat) := \min_{\vv \in \lat \setminus \{ \zero \}} \norm{\vv}.$$
Next, we have the (co)volume $\vol(\lat)$ of a lattice which equals the volume of the quotient space $\spa_\R(\lat) / \lat$. For any basis $\Bm$ of $\lat$ we have:
$$\vol(\lat) := \vol(\spa_\R(\lat) / \lat) = \det(\Bm^\top \Bm)^{1/2}.$$
The volume $\vol(\lat)$ of a lattice can be seen as the inverse density of the lattice, i.e., $\vol(\lat)^{-1}$ is roughly the number of lattice vectors per unit volume. 
In a dense enough lattice of rank $n$ there must always be some short non-zero lattice vector, which is quantified by Minkowski's theorem saying that
$$\lambda_1(\lat) \leq \frac{2\vol(\lat)^{1/n}}{\vol(\ball^n)^{1/n}} \approx \sqrt{2n/\pi e} \cdot \vol(\lat)^{1/n},$$
where $\vol(\ball^n)$ is the volume of a unit ball of dimension $n$. 
In a typical \emph{random lattice} \iffullversion \footnote{One can make this precise by considering a finite measure defined on the space of lattices induced by the Haar measure on $\text{SL}_n(\R)$.}\fi one can give a more precise estimate of the first minimum.
In practice, we usually never have precisely such a random lattice, however, unless they have some special structure, most lattices still satisfy the following simplified heuristic.

\begin{heuristic}[Gaussian Heuristic]\label{heuristic:gh}
Let $\lat$ be a lattice and let $S \subset \spa_\R(\lat)$ be a sufficiently nice body with positive Euclidean volume. Then the expected number of non-zero lattice points in $S$ roughly equals
$| \lat \cap S \setminus \{ 0 \} | \approx \vol_{\spa(\lat)}(S)/\vol(\lat). $ 
In particular, for a lattice of rank $n$ this implies that
$$\lambda_1(\lat) \approx \gh(\lat) := \frac{\vol(\lat)^{1/n}}{\vol(\ball^n)^{1/n}} \approx \sqrt{n/2\pi e} \cdot \vol(\lat)^{1/n}.$$
\end{heuristic} 

\subsection{The shortest vector problem (SVP)} \label{subsec:svp}
One of the main computational lattice problems is that of finding short non-zero lattice vectors. 
Finding a lattice vector $\vv \in \lat$ of minimal length $\norm{\vv} = \lambda_1(\lat)$ is known as the shortest vector problem (SVP). While the best provable algorithm for SVP runs in time and space $2^{n+o(n)}$ in the rank $n$ of the lattice~\cite{aggarwal2015solving,aggarwal2015solvingdet}, heuristic versions can be much better.
The fastest heuristic algorithm for SVP runs in time $2^{0.292n+o(n)}$ and space $2^{0.2075n+o(n)}$~\cite{becker2016new}. We commonly refer to these algorithms as \emph{lattice sieving} algorithms. 
In practice, these algorithms can solve SVP up to dimensions $150$ in reasonable time~\cite{albrecht2019general}, with record computations stretching up to dimension $200$~\cite{ducas2021advanced,zhao2025sieving}.

The above complexities are for finding an (almost) shortest vector in a typical lattice for which $\lambda_1(\lat) \approx \gh(\lat)$.
These seem to be the hardest cases to find a shortest vector both in theory as in practice.
The lattice we will consider has some special structure and therefore does not always fall into this regime.
For example, when the short vector $\vv \in \lat$ we aim to find is shorter than what the Gaussian Heuristic predicts.
Alternatively, even if the lattice satisfies $\lambda_1(\lat) \approx \gh(\lat)$, the short vector $\vv \in \lat$ we aim to find does not always have to be the shortest vector.   

\subsubsection*{The case $\norm{\vv} < \gh(\lat)$.} 
We call a shortest vector $\vv \in \lat$ \emph{unusually short} if $\norm{\vv} = \lambda_1(\lat) < \gh(\lat)$. 
Finding such an unusually short vector is typically easier\iffullversion . In particular, such a short vector can be found\fi ~by running a lattice reduction algorithm such as BKZ~\cite{schnorr1994lattice} which makes polynomially many calls to an exact SVP oracle in dimension $\beta \leq n$\iffullversion ~and thus heuristically runs in time $2^{0.292\beta+o(n)}$\fi . The dimension $\beta$ depends on $n$ and the ratio $\gh(\lat)/\lambda_1(\lat)$ which is called the \emph{gap}. The larger the gap, the lower $\beta$ will be relative to $n$ in order to successfully recover the unusually short vector.

\subsubsection*{The case $\norm{\vv} \geq \gh(\lat)$.}
We consider the case where one wants to find a sufficiently short vector $\vv \in \lat$ for which $\gh(\lat) \leq \norm{\vv} \leq C \cdot \gh(\lat)$ for some $C \geq 1$.
Finding any such short vector is typically easier than finding a shortest vector, again by the use of a lattice reduction algorithm.
In fact, in a typical lattice there are according to the Gaussian Heuristic roughly $C^n$ of such vectors.
However, in our case we will be interested in finding a specific (or a few specific) vectors among those $C^n$ vectors.
Finding a needle in such a haystack will quickly become infeasible, and thus we are typically constrained to quite low values of $C$. The current best heuristic algorithms for SVP naturally compute almost all vectors up to length $\sqrt{4/3} \cdot \gh(\lat)$ and thus in practice we typically constrain ourselves to $C \leq \sqrt{4/3}$ in large dimensions.

\subsubsection*{Dimensions for free}
One can consider the projected lattice $\pi(\lat)$ of rank $n-k$ where $\pi$ projects away from $k$ (primitive) lattice vectors. For a shortest vector $\vv \in \lat$ (with $\norm{\vv} \approx \gh(\lat)$) its projection $\pi(\vv)$ might have $\norm{\pi(\vv)} \leq \sqrt{4/3} \cdot \gh(\pi(\lat))$ and thus $\pi(\vv)$ can be found by running the SVP algorithm on $\pi(\lat)$. Then $\vv$ can be recovered by an appropriate lifting.
One thus decreases the rank of the SVP by $k$ dimensions, which for typical lattices can heuristically be picked as $k = \theta(n / \log(n))$~\cite{ducas2018shortest}. 




%% file: 4-smooth_rat_lat.tex
\section{Generalised Prime Number Lattice}\label{sec:smoothratlat}


Here we reintroduce and generalise the prime number lattice that will be the main subject throughout the rest of this work. We remark that other names have been used for the lattice including the ``Schnorr-Adleman lattice". 


Let $\alpha, \alpha_i \in \R$, for $i \in \{ 1 , \dots , n \}$, and $P = \{ p_i \}_{i=1}^n$ be a finite subset of distinct primes of size $n$. We call the $\alpha_i$'s \emph{weights} and call $P$ a \emph{factor base}. We say that a positive integer is \emph{$P$-smooth} if all its prime divisors are in $P$. 
Note that for $P = P_B = \{ p \leq B \}$ this coincides with being $B$-smooth.

\begin{definition}
	We define the \emph{prime number lattice}  adjoined to $\alpha$, $\alpha_i$ and a factor base $P$, denoted $\lat_{\alpha,\alpha_i,P}$, to be the lattice with the following basis matrix:
	\begin{align*}
	\Bm_{\alpha, \alpha_i, P} \coloneqq \left( \begin{array}{cccc}
	\log(p_1)\alpha & \log(p_2)\alpha  & \cdots & \log(p_n)\alpha \\
	\alpha_1 & 0  & \cdots & 0 \\
	0 & \alpha_2 & \cdots & 0 \\
	\vdots & \vdots & \ddots & \vdots \\
	0 & 0 & \cdots & \alpha_n
	\end{array} \right).
	\end{align*}
	When $P = P_B$ we call this lattice the \emph{full prime number lattice} (or simply the \emph{full lattice}) with respect to the smoothness bound $B$.
\end{definition}

Besides the relevance of this lattice to integer factoring which was alluded to in the introduction, the lattice also plays a role in lattice sphere packings, where it was used to prove the NP-hardness of SVP under randomised reductions \cite{micciancio2002complexity}. Historically one restricts the weights $\alpha_i$ to be either $\log(p_i)$ or $\sqrt{\log(p_i)}$. We write $\alpha$ and $\alpha_i$ abstractly for the time being and make concrete choices later on.

A vector $\Bm_{\alpha, \alpha_i, P} \xv$ in this lattice where $\xv \in \Z^n$ can be associated to a quotient of coprime $P$-smooth integers $a,b$ which we call a \emph{$P$-smooth rational} (or just a \emph{smooth rational}). The concrete smooth rational is apparent in the first entry of the lattice vector: $\alpha \sum_{i=1}^n \log(p_i) x_i = \alpha \log(a/b)$. Moreover, the association goes the other way since $x_i = \mathrm{val}_{p_i}(a/b)$, so we get the following correspondence:

\begin{lemma} \label{lem:genvec}
	For each $\alpha$ and $\alpha_i$, there is a one-to-one correspondence\footnote{A more appropriate name for this lattice might be the ``smooth rational lattice" given this correspondence. But we chose to call it the ``prime number lattice" for historical context.} between $P$-smooth rationals and lattice vectors in $\lat_{\alpha,\alpha_i,P}$ given by:
	$$\frac{a}{b} = \prod_{i=1}^n p_i^{x_i} \in \Q \longleftrightarrow \Bm_{\alpha, \alpha_i, P} \xv \in \lat_{\alpha,\alpha_i,P} .$$
\end{lemma}

The other entries in the lattice vector are $\alpha_i x_i$ for each $i$. 
These help to balance out the contribution of the first entry $\alpha \log(a/b)$ with the size of the exponents $|x_i|$.
\iffullversion If these $\alpha_i$ would be $0$ instead, then $\lat_{\alpha,\alpha_i,P}$ would no longer be discrete (and thus a lattice) and $\log(a/b)$ could be arbitrarily close to $0$, leading to a very short vector without producing a smooth twin. \fi

\subsection{Properties of the prime number lattice}
Due to Lemma~\ref{lem:genvec} the length of each vector can directly be computed from the smooth rational $a/b$ and its rational prime factorisation. Such lengths should be considered relative to the Gaussian Heuristic and thus the \iffullversion (normalised) \fi volume of the lattice. This gives the following lemmas.
\begin{lemma} \label{lem:normvec}
	Let $\vv = \Bm_{\alpha, \alpha_i, P} \xv \in \lat_{\alpha,\alpha_i,P}$ be a lattice vector corresponding to a $P$-smooth rational $a/b = \prod_{i=1}^n p_i^{x_i}$ with $\gcd(a,b) = 1$. Then we have
	\begin{align*}
		\norm{\vv}^2 & = \alpha^2 \cdot \log\left(\frac{a}{b}\right)^2 + \sum_{i=1}^n (x_i \alpha_i)^2.
	\end{align*}
	 Moreover we have the following bounds$$\frac{(\sum_{i=1}^n | x_i \alpha_i |)^2}{n} \leq \norm{\vv}^2 - \alpha^2 \cdot \log\left(\frac{a}{b}\right)^2 \leq \Bigl(\sum_{i=1}^n | x_i \alpha_i |\Bigr)^2.$$ 
\end{lemma}
\begin{proof}
	The first statement follows from \iffullversion the correspondence in \fi Lemma~\ref{lem:genvec}\iffullversion . For the bounds, note that $\norm{\vv}^2 - \alpha^2 \cdot \log\left(\frac{a}{b}\right)^2 = \sum_{i=1}^n (x_i\alpha_i)^2 = \norm{\yv}^2$ where $\yv := (x_i\alpha_i)_{i=1}^n$. The result then \else ~ and the second \fi follows from the general norm-inequality $\frac{1}{n}\norm{\yv}_1^2 \leq \norm{\yv}_2^2 \leq \norm{\yv}_1^2$\iffullversion . \else ~where $\yv := (x_i\alpha_i)_{i=1}^n$. \fi
\end{proof}

\begin{lemma}[Volume] \label{lem:volume}
	We have $$\vol(\lat_{\alpha, \alpha_i, P}) = \sqrt{\alpha^2 \cdot \sum_{i=1}^n (\log(p_i)/\alpha_i)^2+1} \cdot \prod_{i=1}^n \alpha_i.$$
	If $\alpha_i := \log(p_i)$, the above simplifies to $\vol(\lat_{\alpha, \log(p_i), P}) = \sqrt{\alpha^2 n + 1} \cdot \prod_{i=1}^n \log(p_i)$.
\end{lemma}
\begin{proof}
Let $\Cm$ be the basis where each $i$-th column of $\Bm_{\alpha, \alpha_i, P}$ is divided by $\alpha_i$. Then $\vol(\lat_{\alpha, \alpha_i, P}) = \vol(\lat(\Cm)) \cdot \prod_{i=1}^n \alpha_i$. The first row of $\Cm$ is a vector $\cv^\top = (\log(p_i) \cdot \alpha / \alpha_i )_{i=1}^n$, and the remaining $n$ rows form an identity matrix $I_n$. So we have $\Cm^\top \Cm = I_n^\top I_n + \cv \cv^\top$ which has eigenvalues $1+\norm{\cv}^2, 1, \ldots, 1$ and thus determinant $1+\norm{\cv}^2$. So $\vol(\lat_{\alpha, \alpha_i, P}) = \sqrt{1+\norm{\cv}^2} \cdot \prod_{i=1}^n \alpha_i$ from which we obtain the result. 
\end{proof}

\subsection{Finding smooth twins from the prime number lattice} \label{sub:idea}

We give some intuition for finding smooth twins from this lattice.
Recall from Section~\ref{sec:latticeprelims} that one can find a short vector $\vv$ in a lattice $\lat$ if the ratio $\norm{\vv}/\gh(\lat)$ is not too large. 

First we consider the norm of the vector $\vv = \Bm \xv \in \lat_{\alpha, \alpha_i, P}$ corresponding to a smooth rational $a/b$. Lemma~\ref{lem:normvec} says that
$$\norm{\vv}^2 = \alpha^2 \cdot \log\left( \frac{a}{b} \right)^2 + \sum_{i=1}^n (x_i \alpha_i)^2.$$
Observe that for a smooth twin $(a,b) = (r,r+1)$ we have that $a/b$ is very close to $1$ and thus $\log(a/b) \approx 0$, while for other smooth rationals the contribution of $|\log(a/b)|$ can be much larger.
More precisely, for large enough $r$ we have the approximation $\alpha^2 \log(a/b)^2 \approx (\alpha/(r+1))^2$.
From this we see that we can actually pick $\alpha$ up to size $O(r)$ to get $\alpha^2 \cdot \log(a/b)^2 = O(1)$. In other words, we can increase $\alpha$ significantly without increasing the length of $\vv$ much.

For the ratio of the vector $\vv$ relative to the Gaussian Heuristic, we also have to consider the sparsity or equivalently the volume of the lattice.
From Lemma~\ref{lem:volume} we see that this is mainly determined by $\alpha$, in particular increasing $\alpha$ makes the lattice sparser.
As we can freely increase $\alpha$ up to $O(r)$ essentially without increasing the norm of $\vv$ one could hope that the lattice becomes sparse enough for $\vv$ to be a short vector in the lattice. 
To make this reasoning more precise we first give a concrete and \emph{optimal} choice of $\alpha$ in the following proposition.

\begin{proposition} \label{prop:alphaopt} 
	Let $a/b$ be a $B$-smooth rational and $P_{a,b} \coloneqq \{ p \in P_B : p \mid ab \}$. Choose a factor base $P = \{ p_i \} \supseteq P_{a,b}$ with $n = \# P$; and weights $\alpha_i $ for $i = 1, \dots , n$. Set $\beta_1 \coloneqq \log^2(a/b)$ and $\beta_2 \coloneqq \sum_{i=1}^n (x_i \alpha_i)^2$ where $x_i = \mathrm{val}_{p_i}(a/b)$. Then the choice $$\alpha = \alpha_{\mathrm{opt}} \approx \sqrt{\frac{\beta_2}{(n-1)\beta_1}}$$ minimises the GH ratio $\norm{\vv}/\gh(\mathcal{L}_{\alpha,\alpha_i,P})$, among all $\alpha \in \R$, where $\vv \in \mathcal{L}_{\alpha,\alpha_i,P}$ is the corresponding vector to the smooth rational $a/b$. 
\end{proposition}

\begin{proof}
	Consider the following function $f(\alpha) = \norm{\vv}^2/\gh(\mathcal{L}_{\alpha,\alpha_i,P_B})^2$ as a one variable function in terms of $\alpha$. With Lemma~\ref{lem:normvec} and Lemma~\ref{lem:volume}, we approximate\footnote{Two approximations are made here: (1) is to remove the $1$ from Lemma~\ref{lem:volume} which contributes negligibly to the volume; and (2) is the Gaussian heuristic approximation from Heuristic~\ref{heuristic:gh}.} $f(\alpha)$ as
	\begin{align*}
		f(\alpha) \approx \frac{\alpha^{2-2/n}\beta_1}{\gamma^2} + \frac{\beta_2}{\gamma^2\alpha^{2/n}}, \quad \mbox{where } \gamma = \sqrt{\frac{n}{2\pi e}} \left[ \sqrt{\sum_{i = 1}^n \left( \frac{\log(p_i)}{\alpha_i} \right)^2} \prod_{i=1}^n \alpha_i \right]^{\frac{1}{n}} .
	\end{align*}
	Minimising this expression in $\alpha$ is a straightforward calculus exercise. The result of this exercise gives the desired expression for $\alpha_{\mathrm{opt}}$. 
\end{proof}
\begin{corollary} \label{cor:optratio}
	With $P,n,\beta_1,\beta_2$ and $\gamma$ defined in the Proposition~\ref{prop:alphaopt} and the optimal choice $\alpha=\alpha_{\mathrm{opt}}$, we have
	\begin{align*}
		\frac{\norm{\vv}}{\gh(\mathcal{L}_{\alpha,\alpha_i,P})} \approx \left(\sqrt{\frac{\beta_2}{n-1}} \right)^{1-1/n} \cdot \frac{\sqrt{n\beta_1^{1/n}}}{\gamma}. 
	\end{align*}
\end{corollary}
In Section~\ref{sec:analysis} we analyse this expression for $f(\alpha_{\mathrm{opt}})$ and show that, for at least the optimal smooth twins, the corresponding lattice vector is either the shortest vector or one of the shortest. A priori computing this $\alpha_{\mathrm{opt}}$ requires knowledge of the smooth rational itself since the terms use it and their factorisations. However by estimating $\beta_1$ and $\beta_2$ one can choose $\alpha$ approximately. It might not give the absolute shortest vector among all $\alpha$'s but it will still correspond to a short vector. One can also pick several $\alpha$ over some \iffullversion well-informed \fi range in the hope that one is close to $\alpha_{\mathrm{opt}}$. 

The straightforward strategy to find smooth twins is now as follows: first one picks a smoothness bound $B$ and uses the factor base $P_B$. Then one makes a choice of $\alpha$, which is hopefully close to $\alpha_{opt}$ (from Proposition~\ref{prop:alphaopt}), before constructing the full lattice $\lat = \lat_{\alpha, \alpha_i, P_B}$. Then one computes short vectors in $\lat$ using the techniques from~\S\ref{subsec:svp} (either lattice sieving, BKZ reduction and/or dimensions for free depending on the GH ratio). Finally, for each short vector $\vv = \Bm_{\alpha, \alpha_i, P} \xv$, one computes $a = \prod_{i : x_i > 0} p_i^{x_i}$ and $b = \prod_{i : x_i < 0} p_i^{-x_i}$ and checks if $|a-b| = 1$. 



\begin{remark}
	We note a similarity between the CHM algorithm (see\iffullversion ~\S\ref{subsec:consmethod}\else ~Section~\ref{sec:prior}\fi) and our sieving-based SVP solver. Modulo the multiplicative versus additive subtlety (since we work with logarithms), the procedure in the CHM algorithm and lattice sieving correspond to the same process --
	yet the two algorithms take different paths to find large smooth twins. The CHM method starts with small pairs that are already very close and constructs pairs of increasing size. Whereas the lattice siever starts with large pairs that are far apart and reduces their difference.
\end{remark}



While the optimal twins will typically correspond to the shortest vectors\iffullversion , i.e., $\norm{\vv}/\gh(\lat) = \lambda_1(\lat)/\gh(\lat)$ with a ratio less than or close to $1$, \fi ~the same cannot be said about the smaller twins simply because there are more of them. So one will not be able to find these twins by finding the shortest vector in the full prime number lattice. Furthermore, in high dimension when $\norm{\vv}/\gh(\lat) > 1$, and especially above $\sqrt{4/3}$, recovering $\vv$ can quickly become infeasible. Fortunately there are a few some solutions to this problem which we highlight here. 

\subsubsection*{Guessing.} Instead of working with the full lattice to find $B$-smooth twins, one can guess a set $Q$ of prime factors that do not appear in the smooth twin and work with the factor base $P = P_B \setminus Q$. 
The idea is that a $B$-smooth twin typically only uses a small fraction of primes from $P_B$ in its factorisation.
More precisely, with this factor base $P$ we work in a sublattice of the full lattice by removing all vectors that correspond to smooth rationals containing a factor in $Q$. If the guess is correct this decreases both the lattice dimension as well as the ratio $\norm{\vv}/\gh(\lat)$.  

\begin{remark}
	By splitting the factor base into a fixed and a guessing part one could interpret each guess as a Bounded Distance Decoding problem (BDD) in some fixed lattice. Preprocessing~\cite{ducas2020randomized} could allow a low amortised cost for many BDD instances~\cite{karenin2025fast}, which we leave as future work.
\end{remark}

\subsubsection*{Dimensions for free.} We can alternatively reduce the lattice sieving dimension with the dimensions for free technique as explained in Section~\ref{sec:latticeprelims}.
Even if the vector is not unusually short, one could still hope to find it with this technique at a lower probability. 
We will refer to the number of dimensions we decrease the lattice by as the amount of \emph{lifting} we perform.
Just as for the guessing strategy varying the amount of lifting gives a trade-off between success probability and cost. 




%% file: 5-analysis.tex
\section{Analysis}\label{sec:analysis}

We analyse the resulting expression from Corollary~\ref{cor:optratio} to determine the shortness of vectors in the full prime number lattice corresponding to optimal smooth twins. In other words we make a blanket estimate of $r = e^{2e\sqrt{B}}$ for optimal twins (coming from Theorem~\ref{thm:opt}) and analyse $\norm{\vv}/\gh(\lat_{\alpha_{\mathrm{opt}},\alpha_i,P_B})$. We give both theoretical and experimental analyses for various continuous choices of $\alpha_i$ (in terms of $\log(p_i)$). 

Recall the expressions for $\beta_1$, $\beta_2$ and $\gamma$ which are used in the GH ratio as 
\begin{align} \label{eqn:betagamma}
	\beta_1 = \log^2(a/b), \ \beta_2 = \sum_{i=1}^n (x_i \alpha_i)^2 \mbox{ and } \gamma & = \sqrt{\frac{n}{2\pi e}} \Biggl[ \sqrt{\sum_{i = 1}^n \left(\frac{\log(p_i)}{\alpha_i} \right)^2} \prod_{i=1}^n \alpha_i \Biggr]^{\frac{1}{n}},
\end{align}
where with our choice of factor base $P = P_B$ we have $n = \pi(B)$. The expression $\beta_1$ (and more precisely $\beta_1^{1/n}$) is straightforward to analyse. 

\begin{lemma} \label{lem:analone}
	 Assuming the optimal twin estimate for smooth twins, $r = e^{2e\sqrt{B}}$, we have $\beta_1^{1/n} \to 1$ as $n \to \infty$.
\end{lemma}

\iffullversion
\begin{proof}
	Start with this asymptotically more precise statement $\beta_1 = 1/r^2 + O(1/r^3)$.
	Using the Taylor series for $\log{(1+x)}$ for small $x$, we get
	\begin{equation*}
		\log{\beta_1} = \log(r^{-2} (1 + O(1/r)) = -2\log(r) + O(1/r) = -4e \sqrt{B} + O(1/r).
	\end{equation*}
	By the prime number theorem, we find that $B = n \log(n) + O(n\log(\log(n)))$.
	Hence, $\log{(\beta_1^{1/n})} \to 0$ as $n \to \infty$. Thus $\beta_1^{1/n} \to 1$ as $n \to \infty$.
\end{proof}
\else
\begin{proof}
	Start with this asymptotically more precise statement $\beta_1 = 1/r^2 [ 1 + O(1/r) ]$.
	Using the Taylor series for $\log{(1+x)}$ for small $x$ and the prime number theorem, $B = n \log(n) + O(n\log(\log(n)))$, we get $\log{(\beta_1^{1/n})} \to 0$ as $n \to \infty$. 
\end{proof}
\fi

\subsection{Theoretical analysis with $\alpha_i = \log(p_i)$}

Now we deal with analysing $\beta_2$ and $\gamma$. This is non-trivial to do generically, i.e. without an explicit choice of $\alpha_i$. The most challenging aspect in doing this generically is analysing $\beta_2$ in its entirety as well as the sum $\sum_{i=1}^n \left( (\log(p_i)/\alpha_i)^2 \right)$ in $\gamma$. For certain choices of $\alpha_i$ these quantities can be analysed. In particular we make a choice of $\alpha_i = \log(p_i)$ for this analysis.

\begin{lemma} \label{lem:analtwov2}
	Taking $\alpha_i = \log(p_i)$, as $n \to \infty$ we have 
$$\gamma \sim \sqrt{\frac{n}{2\pi e}} \Bigl( \log(n) + \log(\log(n)) - 1 \Bigr).$$
\end{lemma}

\begin{proof}
	With our choice of $\alpha_i$ we can rewrite $\gamma$ from~\eqref{eqn:betagamma} as
	\begin{align*}
		\gamma = \sqrt{\frac{n}{2\pi e}} \left( \sqrt{n} \prod_{j=1}^n \log(p_j) \right)^{1/n}.
	\end{align*}
	By Lemma~\ref{lem:logprod} we have $\left(\prod_{i=1}^n \log(p_i) \right)^{1/n} \sim \log(B) - 1$. Since $n^{1/2n} \to 1$ and (once again) $B \sim n \log(n)$ we obtain the desired result.
\end{proof}

\begin{lemma} \label{lem:analthreev2}
	Taking $\alpha_i = \log(p_i)$ and the optimal twin estimate, we have $$\frac{16e^2B}{n} \leq \beta_2 \leq 4\left( 2e\sqrt{B} + \frac{1}{r} \right)^2.$$
\end{lemma}

\begin{proof}
	Recall that $\beta_2 = \sum_{i=1}^n x_i^2  \log^2(p_i)$ and, by Lemma~\ref{lem:normvec}, we have the inequality
	\begin{align} \label{eqn:beta2logpi}
		\frac{(\sum_{i=1}^n | x_i \log(p_i) |)^2}{n} \leq \beta_2 \leq \Bigl(\sum_{i=1}^n | x_i \log(p_i) |\Bigr)^2.
	\end{align}
	One can calculate that $\sum_{i=1}^n | x_i \log(p_i) | = \log(r)+\log(r+1)$. Using the optimal twin estimate, this is bounded below by $4e \sqrt{B}$ and bounded above by $4e \sqrt{B} + 2/r$. Substituting these bounds in~\eqref{eqn:beta2logpi} gives the desired lower and upper bounds.
\end{proof}


\begin{proposition} \label{prop:shortopt}
	Let $\vv \in \lat = \lat_{\alpha,\alpha_i,P_B}$ be the lattice vector corresponding to an optimal $B$-smooth twin $(r,r+1)$ with the optimal choice $\alpha = \alpha_{\mathrm{opt}}$ and weights $\alpha_i = \log(p_i)$. Then we have
	\begin{align*}
		\frac{\norm{\vv}}{\gh(\lat_{\alpha,\alpha_i,P_B})} = O\left( \frac{1}{\sqrt{\log(n)}} \right) .
	\end{align*}
\end{proposition}

\begin{proof}
	Combining Lemma~\ref{lem:analone} and Lemma~\ref{lem:analtwov2} we get 
	\begin{align*}
		\frac{\sqrt{n\beta_1^{1/n}}}{\gamma} \sim \frac{\sqrt{2\pi e}}{\log(n)+\log(\log(n))-1},
	\end{align*}
	and Lemma~\ref{lem:analthreev2} gives an upper bound for $(\sqrt{\beta_2/(n-1)})^{1-1/n}$ which again by the prime number theorem can be solely expressed in terms of $n$. In particular we have
	\begin{align*}
		\left( \sqrt{\frac{\beta_2}{n-1}} \right)^{1-1/n} & = O \left( \sqrt{\log(n)} \right) .
	\end{align*}
	Combining these together with Corollary~\ref{cor:optratio} gives the intended result. Moreover with the lower bound in Lemma~\ref{lem:analthreev2} we also have $\norm{\vv}/\gh(\lat) = \Omega ( 1/(n\log(n))^{1/2} )$.
\end{proof}


\subsection{Experimental analysis with different $\alpha_i$} \label{subsec:experimentalanalysis}

The theoretical analysis suggests that, for large enough $B$, the lattice vector corresponding to an optimal $B$-smooth twin is unusually short and one should be able to find it with BKZ. However this abstracts the constant in the big $O$ term from the generic bounds (notably from $\beta_2$) and it is unclear from this analysis whether the ratio is less than or greater than $1$ for small $B$. So we conduct experimental analysis to assess their GH ratios $\norm{\vv}/\gh(\lat)$ for \emph{all} $B$-smooth twins. 
We will observe that for the largest twins we are in the scenario where either $\norm{\vv}/\gh(\lat) \approx 1$ or $1 < \norm{\vv}/\gh(\lat) < \sqrt{4/3}$ -- so one needs to do lattice sieving to find them. 


\iffullversion \else We also experiment with more generic weights $\alpha_i = \log^e(p_i)$ which is non-trivial to analyse theoretically primarily due to the more complex $\beta_2$ analysis. \fi
The influence of $\alpha_i$ is considered further in~\S\ref{sub:influence_of_parameters}.
For now we consider the GH ratio $\norm{\vv}/\gh(\lat)$ for several exponents $e$ in $\alpha_i = \log^e(p_i)$ on the set of strictly $199$-smooth twins. 
We think our set of such \iffullversion smooth \fi twins is complete, see Conjecture \ref{conj: 200 smooth twins}, and it should therefore give a good representation.

\begin{figure}[t!]
	\centering
	\begin{subfigure}[b]{0.45\textwidth}
		\includegraphics[width=\textwidth]{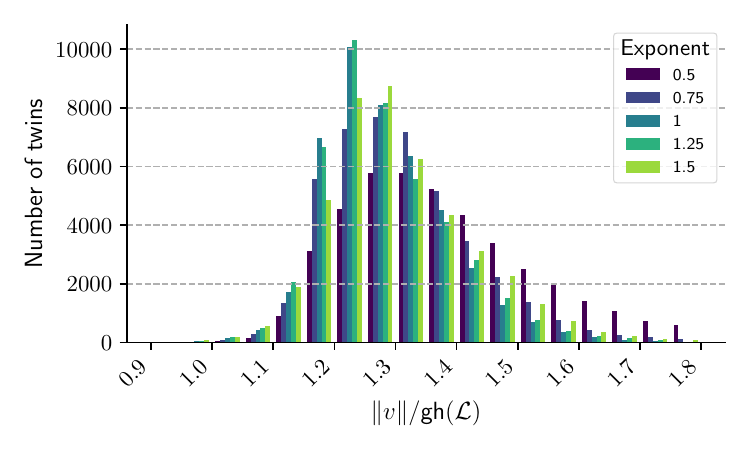}
		\caption{Comparison of the GH ratios achieved for varying exponents $e$ in $\alpha_i = \log^e(p_i)$.}
		\label{subfig:exponent}
	\end{subfigure}
	\begin{subfigure}[b]{0.45\textwidth}
		\includegraphics[width=\textwidth]{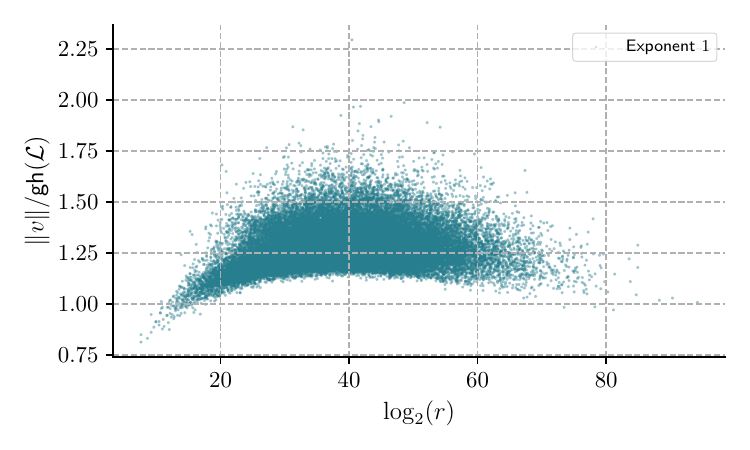}
		\caption{Scatter plot of the GH ratio versus the bit size of the $199$-smooth twins.}
		\label{subfig:scatter}
	\end{subfigure}
	\caption{Plots on the set of \iffullversion all \fi strictly $199$-smooth twins with $\alpha = \alpha_{\mathrm{opt}}$.}
	\label{fig:ratios_199}
\end{figure}

	


In Figure~\ref{subfig:exponent} we see that the exponents $e=1$ or $e=1.25$ lead to the most twins with a small GH ratio.
Among all strictly $199$-smooth twins the ratio $\norm{\vv}/\gh(\mathcal{L}_{\alpha_{\mathrm{opt}},\log(p_i),P_B})$ is on average $1.28$ and no larger than $2.3$. When minimising the ratio with respect to the exponent $e$ (through numerical computations) the average and largest ratio among these twins is $1.25$ and $2.2$ respectively. 
The twins with a larger ratio typically contain an unusually large prime power due to the larger contribution from $(x_i \alpha_i)^2$. For instance this is the case for the $199$-smooth twin $(r,r+1)$ where $r = 107^6 - 1$ which clearly has a large power of $107$. 
In Figure~\ref{subfig:scatter} we see a dependency between the GH ratio and the size of the twins. Notably the ratio for larger twins becomes better, which is in line with our theoretical analysis.

\subsubsection*{Non-continuous weights.} Instead of continuous choices of $\alpha_i$ (in terms of $\log(p_i)$) one might want non-continuous weights: e.g. $\alpha_1 = \log(p_1)/\eta$, for some $\eta > 1$, and $\alpha_i = \log(p_i)$ for $i \geq 2$. This might be necessary when one needs to (say) find smooth twins with a large power of a certain prime\iffullversion ~(e.g. see Section~\ref{sec:isogeny})\fi . We discuss this further in~\S\ref{sub:influence_of_parameters}.

\begin{remark}
	The plots in Figure~\ref{fig:ratios_199} use the optimal $\alpha = \alpha_{\mathrm{opt}}$ for each twin which does not correspond to the same $\alpha$ across the board. 
	For a single $\alpha$ the largest ratio will far exceed $2$. Moreover $\alpha$ influences the size of the resulting twin that one gets as a shortest vector.
\end{remark}

%% file: 6-results.tex
\section{Smooth Twin Results}\label{sec:smoothresults}


We report on results \iffullversion finding smooth twins \fi using the prime number lattice. Experiments ran on the PlaFRIM cluster using CPU nodes ($64$ ZEN2 cores, $2.35$GHz, $256$GB RAM) and GPU nodes ($64$ ZEN3 cores, $2.60$GHz, $2$ NVIDIA A100 GPUs, $512$GB RAM).
We used G6K~\cite{albrecht2019general} and its GPU extension~\cite{ducas2021advanced} for lattice sieving~\footnote{Available at \url{https://github.com/fplll/g6k} and \url{https://github.com/WvanWoerden/G6K-GPU-Tensor} respectively}.
These tools, developed for lattice-based cryptanalysis, have achieved SVP records up to dimension $180$~\cite{ducas2021advanced}.

\subsection{Optimal twins from CPU sieving} \label{sub:cpuresults}

Using heuristic estimates from \S\ref{sub:optimalestimate}, we search the full lattice for large, heuristically optimal twins. 
From Proposition~\ref{prop:alphaopt}, the optimal $\alpha$ for a $B$-smooth twin $(r,r+1)$ is $\alpha_{\mathrm{opt}} \approx \sqrt{\beta_2/(n-1)} \cdot r$. With weights $\alpha_i = \log(p_i)$, $\sqrt{\beta_2/(n-1)} = O(\sqrt{\log(n)})$ (Proposition~\ref{prop:shortopt}).
For an unknown $b$-bit twin, we approximate $\alpha \approx \mu \sqrt{\log(n)}2^b$ for small $\mu$. Experiments suggest $\sqrt{\beta_2/(n-1)}$ remains small for other values $\alpha_i = \log^e(p_i)$.  

For example take $B = 200$, the optimal twin is estimated at $\approx 91.14$ bits (Table~\ref{tab:optsize}). Using $\alpha = 2^{94}$ and weights $\alpha_i = \log(p_i)$ in $\lat_{2^{94},\log(p_i),P_{200}}$, the 6th shortest vector yields a previously unknown $95$-bit $199$-smooth twin. 
\begin{align}
\begin{split} \label{eqn:twin199}
	r & = 3^2 \cdot 5^2 \cdot 7^5 \cdot 11^3 \cdot 59 \cdot 71^2 \cdot 101 \cdot 127 \cdot 173^2 \cdot 197 \cdot 199 \mbox{ and} \\
	r+1 & = 2^{10} \cdot 13 \cdot 17^2 \cdot 23 \cdot 37^2 \cdot 41 \cdot 47 \cdot 61 \cdot 79 \cdot 107^2 \cdot 113^2 \cdot 137.
\end{split}
\end{align}
Exceeding the estimate, this twin is likely optimal.
With $\alpha_{\mathrm{opt}} \approx 2^{96.324}$, this twin corresponds to the 3rd shortest vector in $\lat_{\alpha_{\mathrm{opt}},\log(p_i),P_{200}}$.

\begin{figure}
	\centering
	\resizebox{0.875\textwidth}{!}{	
	\includegraphics{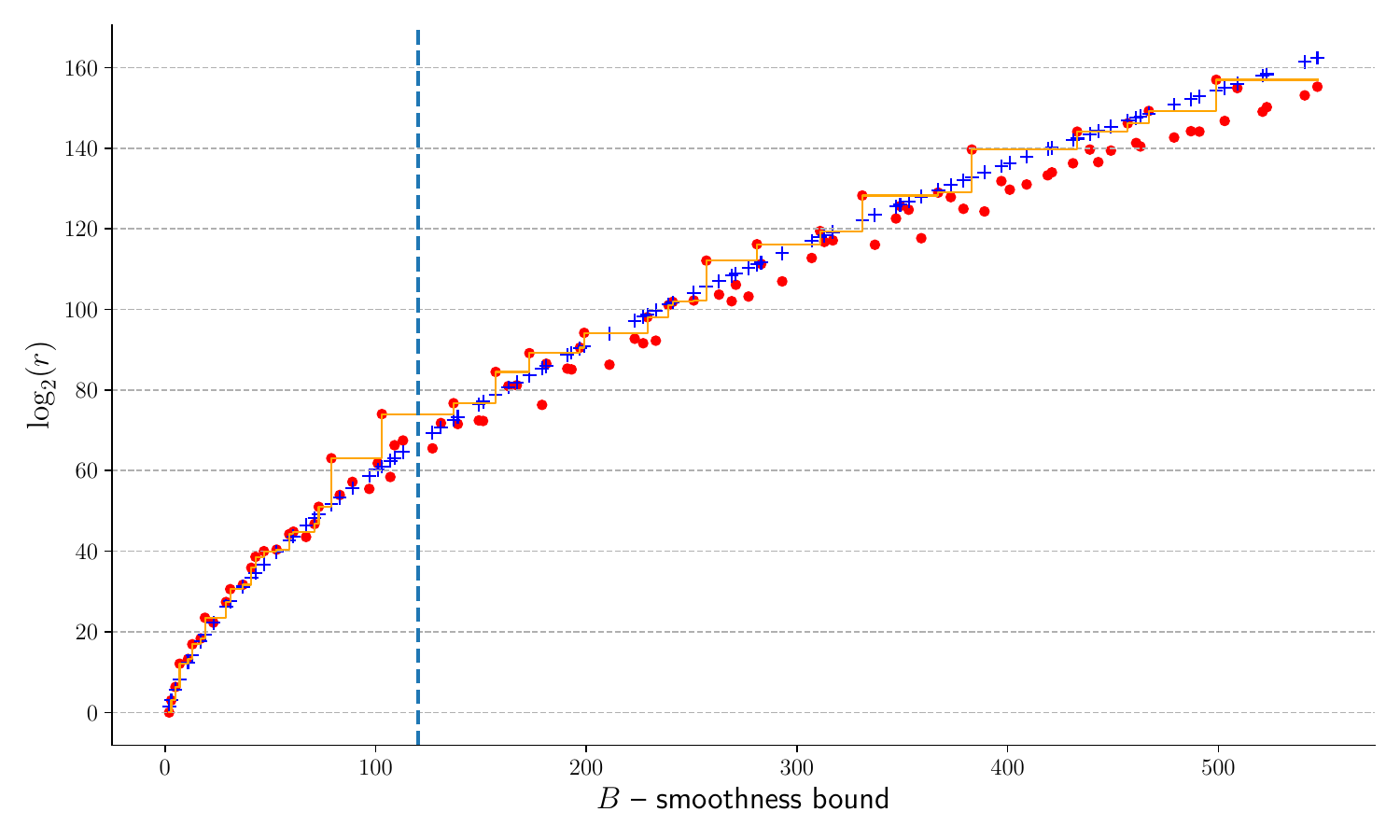}}
	\caption{Comparing the numerical estimates for optimal $B$-smooth twins (marked with a blue plus) against known largest twins with $B \leq 547$. The red dots mark the sizes of the largest \emph{strictly} $B$-smooth twins, while the orange line indicates the cumulative largest twin.
	}
	\label{fig:opttwins547}
\end{figure}

The lattice dimension in this example is $\pi(200) = 46$, well below SVP records. An SVP call costs $2^{0.292\pi(B) + o(\pi(B))}$ (Section~\ref{subsec:svp}), which is significantly less than the $2^{\pi(B)+o(\pi(B))}$ lower bound Pell equations cost (\iffullversion \S\ref{subsec:consmethod} \else Section~\ref{sec:prior}\fi ). 
This saving is the main contributing factor that allows us to \iffullversion work with larger factor bases and \fi find larger optimal twins. \iffullversion We include Figure~\ref{fig:opttwins547} -- an analogue comparison done in Figure~\ref{fig:opttwins113} to cover the largest $B$-smooth twins that we found for $B \leq 547$. We list these twins in Appendix~\ref{app:opt}. \else Figure~\ref{fig:opttwins547} compares the strictly largest $B$-smooth twins against our estimates from \S\ref{sub:optimalestimate}. \fi

Using CPU sieving in the lattice $\lat_{2^{182},\log(p_i),P_{647}}$, we found a conjectured optimal $180$-bit $647$-smooth twin: 
\begin{align*}
\begin{split} 
	r & = 2^4 \cdot 7^3 \cdot 17^2 \cdot 23 \cdot 47 \cdot 61 \cdot 131^4 \cdot 139^2 \cdot 163 \cdot 179 \cdot 197 \cdot 223 \cdot 257 \cdot 293 \\
	& \qquad \cdot 379 \cdot 443 \cdot 563 \cdot 569 \cdot 617 \cdot 647, \mbox{ and} \\
	r+1 & = 3 \cdot 5 \cdot 13 \cdot 19 \cdot 31^2 \cdot 41 \cdot 59 \cdot 71 \cdot 127 \cdot 137 \cdot 151^2 \cdot 233 \cdot 263 \cdot 307^2 \cdot 349^3 \\
	& \qquad \cdot 431 \cdot 467 \cdot 509 \cdot 523 \cdot 601 \cdot 643.
\end{split}
\end{align*}

\subsection{Influence of parameters}
\label{sub:influence_of_parameters}
We examine the influence of the parameters using CPU-based G6K with $B=691$ ($\pi(B) = 125$) and a factor base $P \subseteq P_{691}$. We performed $64$ seeded trials for each parameter set. In the histograms \emph{normalised frequency} represents the total number of unique (unless otherwise stated) smooth twins found divided by $64$.
By default we fix the weights to be $\alpha_i = \log(p_i)^e$ with $e=1$, leave out $k=15$ primes in the guessing phase and lift $l=5$ dimensions with the lifting strategy (sieving dimension $125-k-l=105$). Each trial changes some of these parameters and takes a few hours on $16$ cores.
We study two regimes: \emph{normal} (maximising total number of twins found, $\alpha = 2^{150}$) and \emph{power-of-two} (seeking large powers of two, $\alpha = 2^{128}$). In the latter, we scale $\alpha_1 = \log(2)/\eta$ with $\eta = 20$.

\subsubsection*{Size of found twins ($\alpha$).}
In Figure~\ref{fig:alpha_vs_size} we consider the influence of $\alpha$ on the size of smooth twins $(r,r+1)$ that are found in the normal regime. We varied $\log_2(\alpha)$ from $144$ to $150$ with increments of $2$.
\begin{figure}
	\includegraphics[width=\textwidth]{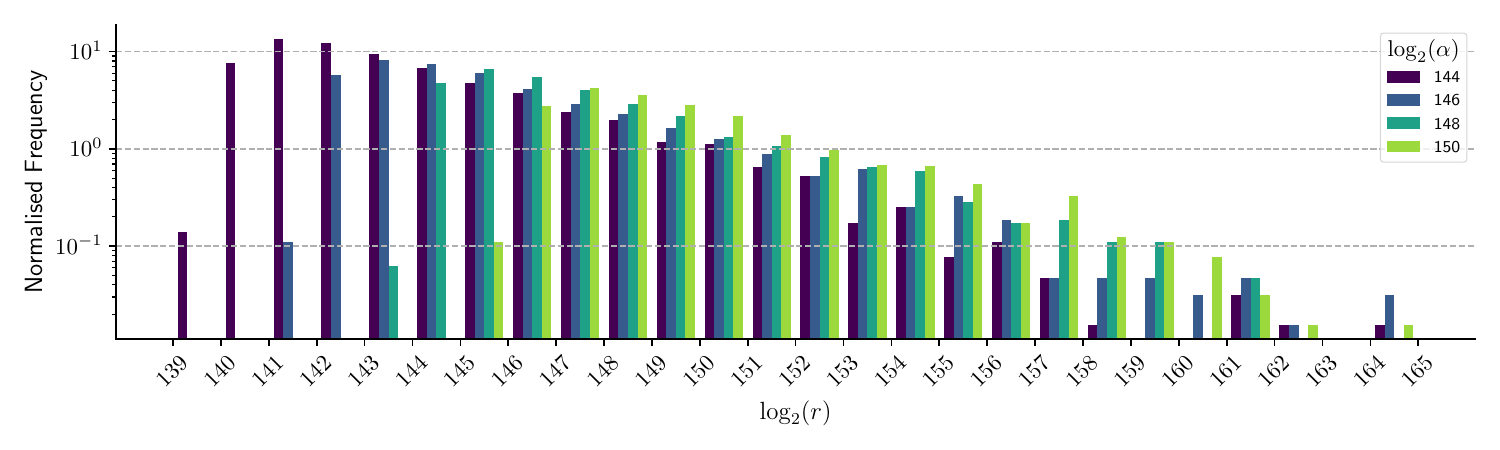}
	\caption{Histogram of the normalised number of smooth twins $(r,r+1)$ found of each bit size versus a varying scalar $\alpha$ in the normal regime.}
	\label{fig:alpha_vs_size}
\end{figure}
Consistent with Proposition~\ref{prop:alphaopt}, twin size scales with $\alpha$, peaking in the interval $\log_2(r) \in [\log_2(\alpha)-3, \log_2(\alpha)-2)$. Larger $\alpha$ yield fewer twins in total, reflecting the scarcity of larger $B$-smooth twins.



\subsubsection*{Diagonal exponent ($e$)}
We varied the exponent $e$ in $\alpha_i = \log(p_i)^e$ from $0.5$ to $1.5$ (step $0.125$) in both regimes. Results are in Figure~\ref{fig:diagexp}.
\begin{figure}
	\includegraphics[width=\textwidth]{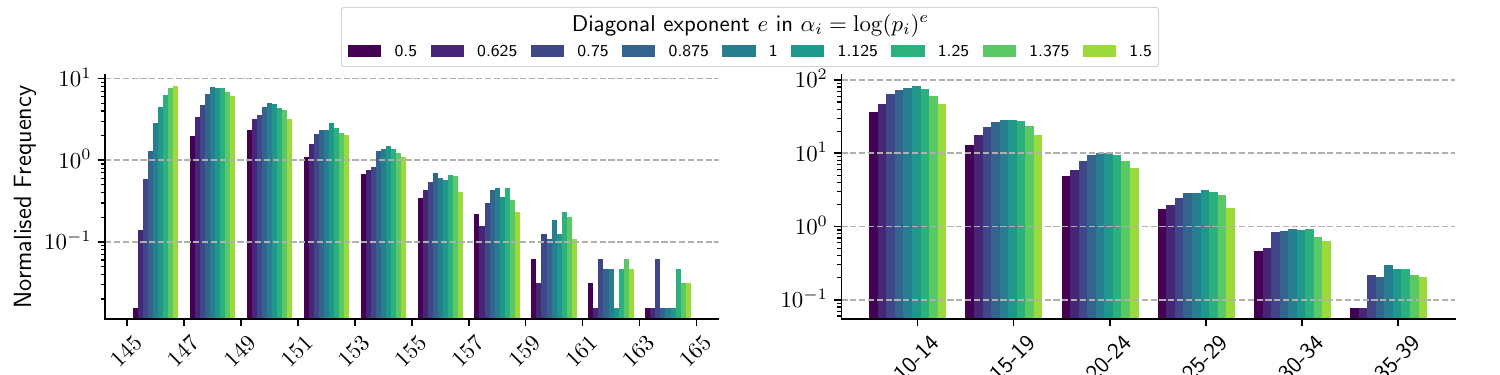}
	\caption{Normalised number of smooth twins found when changing the weight $\alpha_i = \log_2(p_i)^e$. The left (right) plot is in the normal (power-of-two) regime and plotted against the size $\log_2(r)$ (power-of-two exponent $\val_2(r(r+1))$ respectively) of each smooth twin $(r,r+1)$.}
	\label{fig:diagexp}
\end{figure}
In the normal regime (plotting $\log_2(r)$), $1.125 \leq e \leq 1.375$ is optimal, peaking at $e = 1.25$. Larger $e$ slightly reduces the average twin size relative to $\alpha$.
In the power-of-two regime (plotting $\val_2(r(r+1))$), $0.875 \leq e \leq 1.25$ is optimal, peaking at $e = 1.125$.

%

\subsubsection*{Guessing versus lifting ($k,l$).}
We compare guessing ($k$) and lifting ($l$) by varying $k \in \{ 5, 10, 15, 20 \}$ with $k+l = 20$, balancing lattice dimension against success probability\footnote{The case $k=0$ is not considered as otherwise each trial would use exactly the same lattice.}.
\iffullversion Both regimes are considered for this comparison and we make separate plots counting the normalised frequency for the number of found twins with or without duplicates. \fi
In the normal regime (Figure~\ref{subfig:guess_vs_lift_normal}), counting duplicates favors $(10,10)$ or $(15,5)$ over $(5,15)$ or $(20,0)$. For unique twins, guessing dominates, with $(15,5)$ being optimal. Lifting appears to disproportionately find the same unusually short twins repeatedly.
Similarly, in the power-of-two regime (Figure~\ref{subfig:guess_vs_lift_pow2}), $(15,5)$ is optimal for duplicates, while $(20,0)$ performs nearly as well for unique twins.

\begin{figure}
	\begin{subfigure}{\textwidth}
		\includegraphics[width=\textwidth]{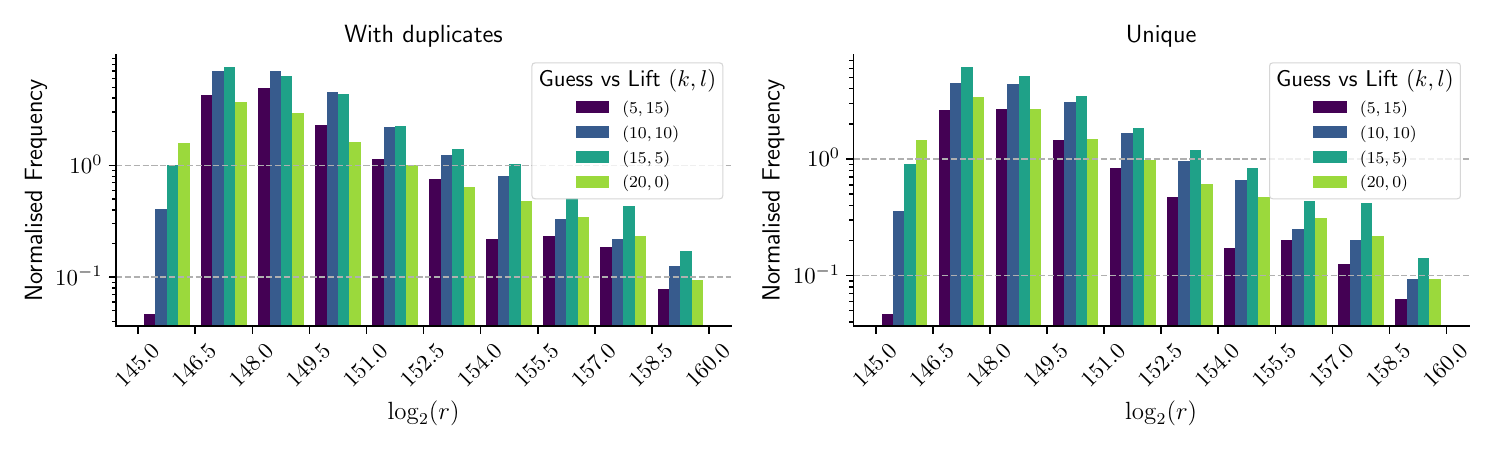}
		\caption{Normal regime.}
		\label{subfig:guess_vs_lift_normal}
	\end{subfigure}
	\begin{subfigure}{\textwidth}
		\includegraphics[width=\textwidth]{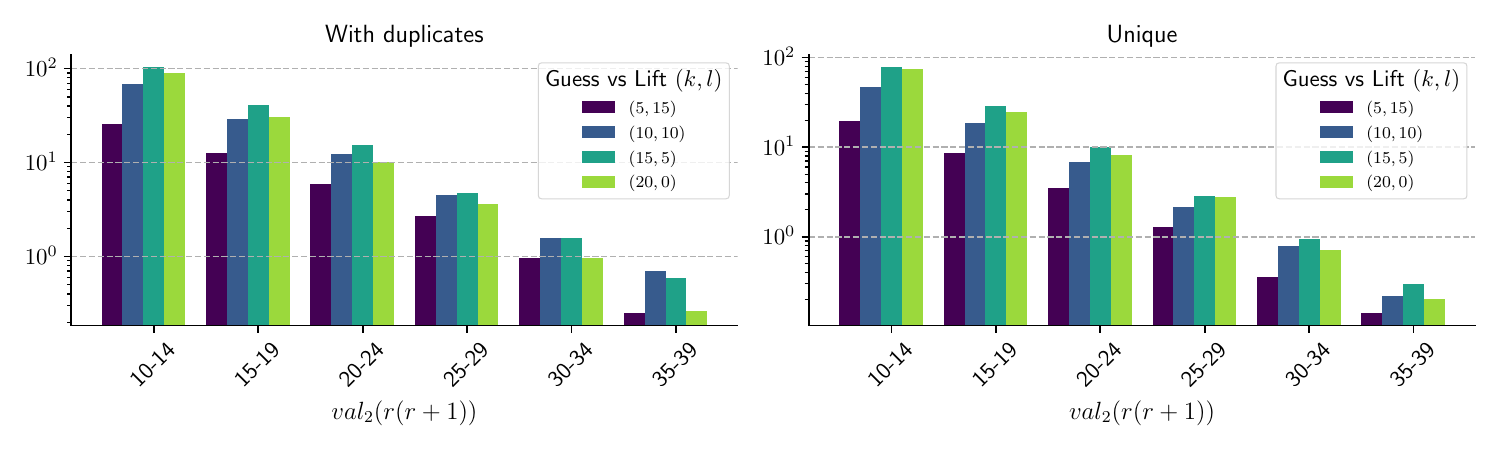}
		\caption{Power-of-two regime.}
		\label{subfig:guess_vs_lift_pow2}
	\end{subfigure}
	\caption{Comparing guessing versus lifting (d4f); varying the guessing and lifting dimensions $(k,l)$ while keeping their sum $k+l=20$ fixed.}
\end{figure}


\subsubsection*{Scaling and exponent trick for finding large powers of two}
To find twins with large powers of two (see Section~\ref{sec:isogeny} for motivation), we test two methods. The first method replaces $p_1=2$ with $p_1=2^f$ ($f > 1$), ensuring found twins satisfy $f | \val_2(r(r+1))$. Increasing $f$ makes the lattice sparser, potentially aiding the search, but misses twins where $f \nmid \val_2(r(r+1))$.
The second method decreases $\alpha_1$ by a factor $\eta > 1$, reducing the norm penalty for large powers of two. This makes the lattice denser and requires balancing the choice of $\eta$. We test $\alpha_1 = \log(p_1)/\eta$ for $\eta \in \{ 1, 5, 10, 20, 40 \}$.


\begin{figure}
	\includegraphics[width=\textwidth]{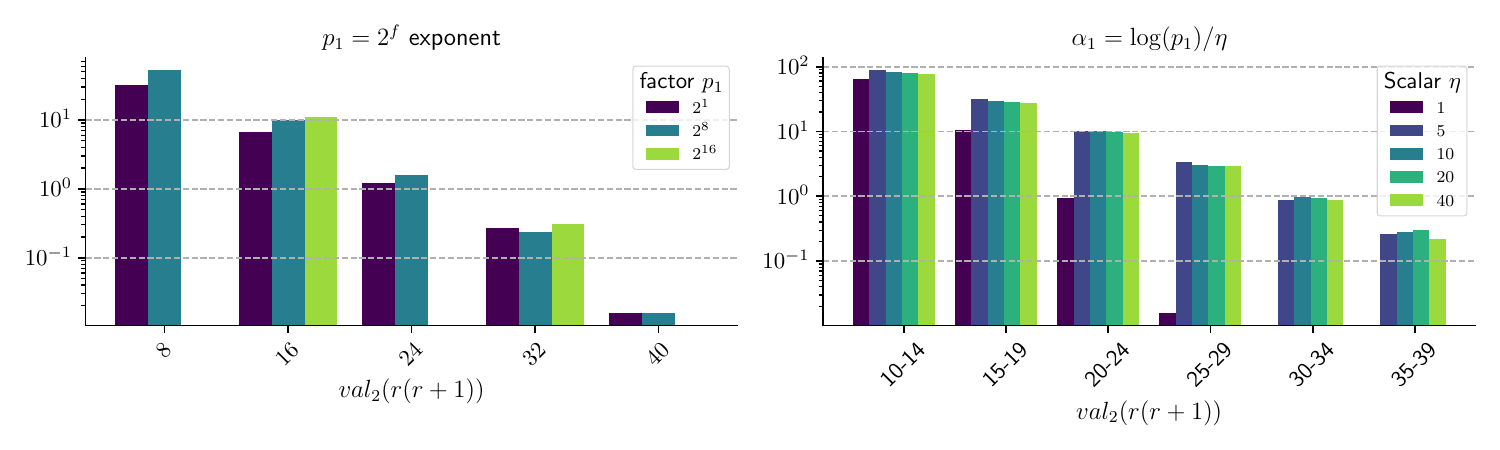}
	\caption{Influence of the power-of-two exponent $f$ and the scalar $\eta$ on finding smooth twins with a large power of two.}
	\label{fig:pow2_scaling_exponent}
\end{figure}

Results in Figure~\ref{fig:pow2_scaling_exponent} show that using $p_1=2^f$ yields more twins with $f | \val_2(r(r+1))$. Downscaling $\alpha_1$ with $\eta=5$ finds twins with $\val_2(r(r+1)) \geq 35$, whereas $\eta=1$ finds essentially none $\geq 25$. Increasing $\eta$ beyond $5$ has little effect.


\subsection{Results from GPU lattice sieving} \label{sub:gpuresults}
We now detail our GPU experiments. The found twins are available at~\cite{zenodo}. Memory constraints limited the effective sieving dimension to $137$. Sieving at this dimension takes half a day on one PlaFRIM node.
With this, using $B = 773$ ($\pi(B) = 137$), we found a $195$-bit $743$-smooth twin \eqref{eqn:opt743} and a $196$-bit $751$-smooth twin \eqref{eqn:opt751}. Both align with heuristic optimality estimates.
\begin{align} \label{eqn:opt743}
\begin{split}
	r & = 7 \cdot 19^3 \cdot 23 \cdot 37 \cdot 47^2 \cdot 79 \cdot 89 \cdot 97^3 \cdot 103 \cdot 211 \cdot 229^2 \cdot 271^2 \cdot 283 \cdot 307 \\
	& \qquad \cdot 401^2 \cdot 503 \cdot 541 \cdot 557 \cdot 631 \cdot 643, \mbox{ and} \\
	r+1 & = 2^{11} \cdot 5^2 \cdot 11 \cdot 13 \cdot 43^2 \cdot 53 \cdot 59 \cdot 131 \cdot 137 \cdot 149 \cdot 163 \cdot 223^3 \cdot 281 \cdot 313 \\
	& \qquad \cdot 337 \cdot 353 \cdot 409 \cdot 449 \cdot 547 \cdot 563 \cdot 653 \cdot 709 \cdot 743.
\end{split}
\end{align}


For larger $B > 773$ ($\pi(B) > 137$), we use a combination of lifting and guessing. For instance with $B = 823$ and lifting $l = 5$ dimensions, we found a $200$-bit $809$-smooth twin (estimated optimal size: $202$ bits).
\begin{align*}
	r & = 2^6 \cdot 17^3 \cdot 23 \cdot 29^2 \cdot 47 \cdot 53^2 \cdot 59 \cdot 67^2 \cdot 107 \cdot 131 \cdot 137 \cdot 163 \cdot 193 \cdot 271 \\
	& \qquad \cdot 277 \cdot 283 \cdot 313 \cdot 359 \cdot 397 \cdot 433 \cdot 523 \cdot 751 \cdot 787 \cdot 809, \mbox{ and} \\
	r+1 & = 3^{13} \cdot 5^2 \cdot 7^2 \cdot 11 \cdot 19 \cdot 41^2 \cdot 71 \cdot 73 \cdot 89 \cdot 139 \cdot 179 \cdot 263^2 \cdot 269 \cdot 557^2 \cdot 577 \\
	& \qquad \cdot 587^2 \cdot 607 \cdot 677 \cdot 683 \cdot 709 \cdot 733.
\end{align*}
This twin can also be found with guessing $k=5$. One should remove a subset $Q \subseteq P_B$ devoid of factors of $r(r+1)$ and small primes. The success probability for choosing $Q \subseteq P_B \setminus P_{300}$ with $\# Q = 5$ is $\approx 0.3224$.


Increased lifting and guessing finds larger twins but potentially sacrifices optimality. For instance for $B = 997$ ($\pi(B) = 168$), we ran experiments for varying $k + l = 168-137 = 31$ and $k > l$ following \S\ref{sub:influence_of_parameters}. With $k = 18, l = 11$, we found a $213$-bit $997$-smooth twin \eqref{eqn:twin997}, which is far from optimal.
Experiments with larger $B$ and increased lifting/guessing yielded no larger twins.



\subsubsection*{Cryptographic smooth twins.} 
Isogeny-based cryptography typically requires $256$-bit twins. While Section~\ref{sec:isogeny} discusses relaxing this requirement, we address here the challenge of finding such twins.
Table~\ref{tab:optsize} suggests needing $B \approx 1250$ ($\pi(B)=204$) for $256$-bit twins. Current SVP records reach dimension $200$ (with lifting)~\cite{albrecht2019general,ducas2021advanced,zhao2025sieving}\iffullversion \footnote{For SVP records see \url{https://www.latticechallenge.org/svp-challenge/}.}\fi, putting this just beyond reach. Even if one incorporates lifting and guessing, one would need a large amount of trials in order to find the desired smooth twin -- adding to the computational cost.

\subsection{Towards enumerating all $B$-smooth twins}

The main obstruction to finding all $B$-smooth twins with our lattice approach is that not all smooth twins correspond to short vectors in the full lattice (see Figure~\ref{fig:ratios_199}). To find them through SVP instances a sufficient (and possibly large) amount of guessing is required (we discuss this further in Remark~\ref{remark:alltwins}).

Alternatively one can use the lattice approach to add previously undiscovered twins to a known set of $B$-smooth twins. One still needs to incorporate guessing to find the obscure twins. If the known set is large then one can minimise the amount of necessary searching and plausibly conjecture the complete set of $B$-smooth twins. 

A large proportion of $B$-smooth twins for $B \leq 547$ was found using the CHM algorithm~\cite{bruno2022cryptographic} (\iffullversion \S\ref{subsec:consmethod}\else Section~\ref{sec:prior}\fi). Notably the larger $B$-smooth twins are not found with the CHM algorithm, most of which our SVP solver should find with minimal guessing. We extensively searched for $B$-smooth twins with $B < 400$ and found many new twins. In particular we found 25 new $200$-smooth twins including the twin in \eqref{eqn:twin199}. We did further experiments to find more $200$-smooth twins with large guessing amounts. After 1-2 weeks of effort no further twins were found. So we conjecture to have the complete set of $200$-smooth twins as summarised in Conjecture~\ref{conj: 200 smooth twins}, and available at~\cite{zenodo}.

No formal claims are made for larger $B$ but we note down in Table~\ref{tab:numsmooth} the cumulative number of $B$-smooth twins found in this and prior work. We did not extend our experiments to $B = 547$ due to the increasing computational effort.
\begin{conjecture}
\label{conj: 200 smooth twins}
	There are exactly 348,865 many $200$-smooth twins.
\end{conjecture}

\begin{table}
	\centering
	\renewcommand{\tabcolsep}{0.2cm}
	\renewcommand{\arraystretch}{1.15}
		\begin{tabular}{cccc}
			& CHM\textsubscript{$B$} & CHM\textsubscript{$547$} & SVP \\
			\hline
			$B = 150:$ & 73,694 & 74,006 \emph{(+312)} & 74,007 \emph{(+1)} \\
			$B = 200:$ & 346,192 & 348,840 \emph{(+2,648)} & 348,865 \emph{(+25)} \\
			$B = 250:$ & 826,750 & 835,613 \emph{(+8,863)} & 835,880 \emph{(+267)} \\
			$B = 300:$ & 2,316,631 & 2,350,100 \emph{(+33,469)} & 2,353,597 \emph{(+3,497)} \\
			$B = 350:$ & --- & 5,431,970 & 5,456,952 \emph{(+24,982)} \\
			$B = 400:$ & --- & 11,840,978 & 11,980,453 \emph{(+139,475)}
		\end{tabular}
	\vspace{1ex}
	\caption{Cumulative number of $B$-smooth twins found. The subscript in CHM\textsubscript{$B$} denotes applying the CHM algorithm with $B$ as input.}
	\label{tab:numsmooth}
\end{table}


\begin{remark} \label{remark:alltwins} 
By Theorem~\ref{thm:opt}, the number of distinct prime factors in $B$-smooth twins is heuristically at most $O(\pi(4e\sqrt{B})) = O(\sqrt{B}/\log(B))$ which for large $B$ is only a small fraction of all $O(B/\log(B))$ primes up to $B$.
Therefore, by doing a very large amount of guessing, with say $k = \pi(B) - O(\sqrt{B}/\log(B))$, we could use the SVP approach alone to find all $B$-smooth twins. Trying all ${\pi(B) \choose k}$ subsets $P \subseteq P_B$ of size $n-k$ and polynomially many distinct $\alpha$'s, the resulting complexity turns out to be subexponential in $B$. Despite this potential gain, it is only practical for very large $B$ and not the small $B$ that we have worked with.
A similar heuristic analysis can be made with the Pell equation approach.
For all square-free discriminants $D \leq e^{4e \sqrt{B}}$, 
one would early abort solving the equation $x^2-Dy^2=1$ using the continued fraction of $\sqrt{D}$, namely when $x$ becomes larger than $2e^{2e\sqrt{B}}$. This should also give a subexponential complexity, but which only pays off for huge $B$.
\end{remark}




%% file: 7-isogeny_impact.tex
\section{Applications}\label{sec:isogeny}


\subsection{SQIsign parameters} \label{subsec:sqisign}

\iffullversion SQIsign is a family of post-quantum signature schemes based on the Deuring correspondence. The high-level idea in each variant is the same but the algorithmic tools in practice differ. For accessible resources on SQIsign and its variants we refer the reader to~\cite{SQISign,luca2022new,leroux2022quaternion,chavez2023sqisign,corte2024apressqi,basso2024sqisign2d,nakagawa2024sqisign2d,NISTPQC-ADD-R2:SQIsign25}. \fi 

%

We revisit the original version of SQIsign\iffullversion \else \footnote{For accessible resources on SQIsign and its variants we refer to~\cite{SQISign,luca2022new,leroux2022quaternion,chavez2023sqisign,corte2024apressqi,basso2024sqisign2d,nakagawa2024sqisign2d,NISTPQC-ADD-R2:SQIsign25}.}\fi, which we call SQIsign1D, and find new parameters. In particular one works with a prime $p$ such that $2^f \cdot T \mid p^2-1$ where $f$ is as large as possible and $T \approx p^{5/4}$ is odd and $B$-smooth. This condition implies that not all of $p^2-1$ needs to be smooth, which differs from the fully smooth twin scenario\iffullversion, i.e. a smooth twin $(r,r+1)$ with $p = 2r + 1$ so that $p^2-1 = 4r(r+1)$\fi . 

\iffullversion 
The main idea in this prime search is to adopt a \emph{boosting} strategy. For $256$-bit primes one uses the polynomial $p_2(x) = 2x^2 - 1$~(emblematic of the ideas in~\S\ref{subsec:probmethod}). For such a $p = p_2(r)$ we need smooth divisors in $r-1,r,r+1$ since $p^2 - 1 = 4r^2(r-1)(r+1)$. There are two ways of choosing $r$ that maximises the chances of meeting all requirements. In both cases one ensures that $r$ is smooth and has a large power of two (which is amplified by the squaring in $p_2+1$).
\else
The main idea in this prime search is to adopt a \emph{boosting} strategy and search for a prime of the form $p = p_2(r) = 2r^2 - 1$ \iffullversion (emblematic of the ideas in~\S\ref{subsec:probmethod}) \fi for a smooth $r$. As $p^2 - 1 = 4r^2(r-1)(r+1)$ we need smooth divisors in $r-1,r+1$. We also ensure that $r$ has a large power of two (which is amplified by the squaring in $p_2+1$). 
\fi

\subsubsection*{Sieve-and-boost approach.} Instead of doing an exhaustive search on smooth $r$ which is expensive, one can do a tailored exhaustive search over integers of the form\footnote{The inclusion of the power of three is not mandatory but it simplifies certain steps in the signing procedure and was a design choice made in the original NIST submission.} $r = 2^{f'} \cdot 3^{g'} \cdot m$ where $m$ is smooth. Then for each such $r$ factorise the smooth parts of $r-1$ and $r+1$ and see if this gives the necessary amount of smoothness. This strategy was adopted for the first round of the NIST submission~\cite{chavez2023sqisign}\iffullversion ~and found the prime $p = p_2(r)$ where $r = 2^{37} \cdot 3^{18} \cdot 2053899652631121509$ with factorisations 
\begin{align*}
	p+1 & = 2^{75} \cdot 3^{36} \cdot 23^{2} \cdot 59^{2} \cdot 101^{2} \cdot 109^{2} \cdot 197^{2} \cdot 491^{2} \cdot 743^{2} \cdot 1913^{2}, \mbox{ and}  \\
	p-1 & = 2 \cdot 7^{4} \cdot 11 \cdot 13 \cdot 37 \cdot 89 \cdot 97 \cdot 107 \cdot 131 \cdot 137 \cdot 223 \cdot 239 \cdot 383 \cdot 389 \cdot 499 \\
	& \qquad \cdot 607 \cdot 1033 \cdot 1049 \cdot 1193 \cdot 1973 \cdot {\color{lightgray}32587069} \cdot {\color{lightgray}275446333} \\
	& \qquad \cdot {\color{lightgray}1031359276391767}.
\end{align*}
The factors highlighted in {\color{lightgray} grey} are the non-smooth factors which are not used in SQIsign1D. A more in-depth description of this strategy is given in~\cite{santos2024finding}. \else . \fi

\subsubsection*{Twin-and-boost approach.} 
If $(r,r \pm 1)$ is a smooth twin with $2^{f'} \mid r$ and $p = p_2(r)$, then $p^2-1$ has a smooth factor $4r^2(r \pm 1) \approx p^{3/2}$. If $f = 2f'+1 \leq \log_2(p)/4$ then $2^f \cdot T  \leq p^{3/2}$; so one simply needs to check primality of $p$. For larger $f$ one requires more small prime factors in $r \mp 1$. Depending on how much larger $f$ is relative to $\log_2(p)/4$, the probability that one has these factors is not too small~\cite{banks2006integers}. 

So a sufficiently good \emph{smooth twin oracle} \iffullversion producing smooth twins with a large power of two \fi can \iffullversion efficiently \fi give SQIsign1D parameters. This observation was first noted in~\cite{bruno2022cryptographic} but they were unable to instantiate the oracle with the CHM algorithm. On the other hand we can instantiate this properly with our lattice approach aimed at finding smooth twins with a large power of two (see~\S\ref{sub:influence_of_parameters}). In particular we found the exceptional prime $p = p_2(r)$ \iffullversion where $r = 2^{31} \cdot 2493490582368659543466244025$ \fi with factorisations 
\begin{align*}
	p+1 & = 2^{63} \cdot 5^4 \cdot 23^4 \cdot 67^2 \cdot 71^2 \cdot 73^2 \cdot 89^2 \cdot 113^2 \cdot 137^4 \cdot 163^2 \cdot 229^2 \cdot 263^2 \cdot 293^2, \mbox{ and}  \\
	p-1 & = 2 \cdot 3 \cdot 7^2 \cdot 11 \cdot 13^2 \cdot 31 \cdot 47^2 \cdot 79^2 \cdot 103 \cdot 151 \cdot 241^2 \cdot 353 \cdot 367 \cdot 389 \cdot 449 \\
	& \qquad \cdot 463 \cdot 499 \cdot {\iffullversion \color{lightgray} \fi 50355971} \cdot {\iffullversion \color{lightgray} \fi 1032403334060991048097384477}.
\end{align*}
With a significantly smaller smoothness bound (on the smooth part of $p^2-1$) compared to the \iffullversion previous parameter \else parameters in~\cite{chavez2023sqisign}\fi , this offers a better signing performance in accordance to the SQIsign1D estimator sage script provided in~\cite{santos2024finding}. In particular one could theoretically get a modest 19.85\% improvement for signing\iffullversion ~but we remark that practical considerations are needed in order to materialise this improvement\fi. However we note that this will not compete with \iffullversion the latest versions of SQIsign and notably \fi the second round submission~\cite{NISTPQC-ADD-R2:SQIsign25} to NIST's signature call. So we present this prime to expand diversity of SQIsign1D parameters.

\iffullversion
\begin{remark}
	SQIsign1D suffers from a bad performance scaling primarily due to the parameter search. In particular $\fpnx{n}$ for $n = 4,6$ was used to get sufficiently practical larger parameters~\cite{chavez2023sqisign} rather than $p_2(x)$. Using $p_2(x)$ would eventually give better parameters but at a larger sieving cost. On the other hand, one can use our twins with twin-and-boost and $p_2(x)$ to get better larger parameters. However we did not find enough twins with a large enough power of two to instantiate the oracle and get practical parameters. 
\end{remark}


\begin{remark}
	In the context of AprèsSQI~\cite{corte2024apressqi} (an adaptation of SQIsign1D which works with larger extensions of $\F_{p^2}$) one can also use the twin-and-boost approach to parameters with a larger $f$. One such parameter with a larger power of two is the prime $p = p_2(r)$ found from a smooth twin $r = 2^{42} \cdot 38345988732608448329403105$.
\end{remark}

\begin{remark}
	As noted in~\cite{santos2024finding} the encryption scheme POKÉ~\cite{basso2025poke} could benefit from $p^2-1$ parameters. One requires a larger power of two compared to SQIsign1D which imposes a larger smoothness bound on the twin when using the twin-and-boost approach. So finding such parameters is challenging with our sieving based algorithm and we leave these considerations as part of future work.
\end{remark}
\fi

%% file: supplementary_material.tex
\section{Asymptotics for the geometric mean of the logarithmic primes}

We provide an elementary proof of the following statement. The proof is adapted discussion from a stack exchange post\footnote{\url{https://math.stackexchange.com/q/679344}} in order to get an asymptotically precise statement. To our knowledge this result is not publicly recorded.

\begin{lemma} \label{lem:logprod}
	As $x \to \infty$ we have
	\begin{align*}
		\left( \prod_{p \leq x} \log(p) \right)^{\frac{1}{\pi(x)}} = \log(x) - 1 - o(1),
	\end{align*}
	where the product is over the primes less than $x$.
\end{lemma}

\begin{proof}
	We consider $\sum_{p \leq x} \log(\log(p))$ which by Abel summation is 
	\begin{align*}
		\sum_{p \leq x} \log(\log(p)) = \pi(x) \log(\log(x)) - \int_2^x \frac{\pi(t)}{t\log(t)} dt.
	\end{align*}
	Using the prime number theorem in the integral term we have
	\begin{align*}
		\int_2^x \frac{\pi(t)}{t\log(t)} dt & = \int_2^x \frac{t/\log(t) + O(t/\log^2(t))}{t\log(t)} dt \\
		& = \frac{x}{\log^2(x)} + O\left( \frac{x}{\log^3(x)} \right) = \frac{\pi(x)}{\log(x)} + O\left( \frac{\pi(x)}{\log^2(x)} \right) .
	\end{align*}
	So we have 
	\begin{align*}
		\sum_{p \leq x} \log(\log(p)) = \pi(x)\left( \log(\log(x)) - \frac{1}{\log(x)} - O\left( \frac{1}{\log^2(x)} \right) \right) 
	\end{align*}
	and exponentiating gives
	\begin{align*}
		\prod_{p \leq x} \log(p) = \left( \log(x) \exp \left\{- \frac{1}{\log(x)} - O\left( \frac{1}{\log^2(x)} \right) \right\} \right)^{\pi(x)} .
	\end{align*}
	Finally, using the elementary fact that $e^{f(x)} = 1 + f(x) + o(1)$ when $f(x) = o(1)$, we obtain the intended result. 
\end{proof}

\iffullversion

\section{Largest known $B$-smooth twins for $113 < B < 500$} \label{app:opt}

In Table~\ref{tab:listlargestwins} we list the known largest $B$-smooth twins for $113 < B < 500$, stating where they were found and whether or not we think they are optimal. 

\begin{table}[p]
\centering
\renewcommand{\tabcolsep}{1.1cm}
\resizebox{\textwidth}{!}{
	\begin{tabular}{cccc}
		$B$ & $r$ & Where & Optimality confidence \\
		\hline
		$127$ & $\mathtt{0x2e2c7fc1e68254960}$ & \cite{bruno2022cryptographic} & Strong \\
		$131$ & $\mathtt{0xdf0971aec4e60bb59f}$ & \cite{bruno2022cryptographic} & Strong \\
		$137$ & $\mathtt{0x1a4e4d2d8a05fd236b45}$ & \cite{bruno2022cryptographic} & Strong \\
		$139$ & $\mathtt{0xbcc8805bbc1372e47b}$ & Here & Strong \\
		$149$ & $\mathtt{0x15bf7b64f95196ee2f9}$ & \cite{bruno2022cryptographic} & Strong \\
		$151$ & $\mathtt{0x141e24a43fd1c03486f}$ & \cite{bruno2022cryptographic} & Strong \\
		$157$ & $\mathtt{0x1670f5d1ce32e4d3879cd3}$ & \cite{bruno2022cryptographic} & Strong \\
		$163$ & $\mathtt{0x1f3f1ccb92039e65d1338}$ & Here & Strong \\
		$167$ & $\mathtt{0x242216734de5ff40fcfeb}$ & \cite{bruno2022cryptographic} & Strong \\
		$173$ & $\mathtt{0x235282118c1597288738b81}$ & Here & Strong \\
		$179$ & $\mathtt{0x13c15bf1e12b1d917b3f}$ & \cite{bruno2022cryptographic} & Strong \\
		$181$ & $\mathtt{0x5c1e1cbfdbe35d79fd2a3f}$ & \cite{bruno2022cryptographic} & Strong \\
		$191$ & $\mathtt{0x27bf7daa74ab9511295150}$ & Here & Strong \\
		$193$ & $\mathtt{0x2276a89f79d543f95c3fff}$ & Here & Strong \\
		$197$ & $\mathtt{0x5b6a5f8dd9d325c38d5407f}$ & Here & Strong \\
		$199$ & $\mathtt{0x48cc28fd91925f14fabdfbff}$ & Here & Strong \\
		$211$ & $\mathtt{0x4dac9441b80a2e417d6718}$ & \cite{bruno2022cryptographic} & Conjectured \\
		$223$ & $\mathtt{0x1a7ff405f4bcbdc136785ed8}$ & Here & Conjectured \\
		$227$ & $\mathtt{0xc2024a68a4f0cbfad153214}$ & \cite{bruno2022cryptographic} & Conjectured \\
		$229$ & $\mathtt{0x42d237ebe8804683635fcfbe6}$ & Here & Conjectured \\
		$233$ & $\mathtt{0x12e5a0f3787fca121abaf380}$ & Here & Conjectured \\
		$239$ & $\mathtt{0x21189a77fc9dbc410d03f5ebd4}$ & Here & Conjectured \\
		$241$ & $\mathtt{0x3a6a0e89727bd5207f05773800}$ & Here & Strong \\
		$251$ & $\mathtt{0x4acb78e14ff5ab674fca758f00}$ & Here & Conjectured \\
		$257$ & $\mathtt{0x1121f079ed6e2c422f4df4a3a8297}$ & Here & Strong \\
		$263$ & $\mathtt{0xcc21a0cbe97a991348eed51472}$ & Here & Conjectured \\
		$269$ & $\mathtt{0x41527dd1d202cc8c89a61bb520}$ & Here & Conjectured \\
		$271$ & $\mathtt{0x455d39041f2d5f0533bc070dbff}$ & Here & Conjectured \\
		$277$ & $\mathtt{0x9204d1ce9c2456089052e0a199}$ & Here & -------- \\
		$281$ & $\mathtt{0x11e3943e7a1b82a6c30f0f31ddffff}$ & Here & Strong \\
		$283$ & $\mathtt{0x9be280ede7efda1f206b7ca5b312}$ & Here & Conjectured \\
		$293$ & $\mathtt{0x7bf7d45de9de5173c97621f1be0}$ & Here & -------- \\
		$307$ & $\mathtt{0x1b0805107c686a2c5ac16505948d7}$ & Here & Conjectured \\
		$311$ & $\mathtt{0xac4336c173cb3dcdb21812fe18a42f}$ & Here & Conjectured \\
		$313$ & $\mathtt{0x1ac04f7ef5aa7a829fffc39ece488e}$ & Here & Strong \\
		$317$ & $\mathtt{0x22b6b48a85f96f1ea35da4c6345139}$ & Here & Conjectured \\
		$331$ & $\mathtt{0x130dd6ecb2c141d788ec6d287806c3c07}$ & Here & Strong \\
		$337$ & $\mathtt{0x10765db8b6a7e23a40f9dbea62188f}$ & Here & -------- \\
		$347$ & $\mathtt{0x5e317ad3faa097d352a5dd7015f5c9f}$ & Here & Conjectured \\
		$349$ & $\mathtt{0x32216c77f71cfec5b2315604df094000}$ & Here & Conjectured \\
		$353$ & $\mathtt{0x1aa35c1f4e8ab15fed2e67bf8eaed72f}$ & Here & Conjectured \\
		$359$ & $\mathtt{0x32bd5b0f00a1b93615c994b43c93d9}$ & Here & -------- \\
		$367$ & $\mathtt{0x1fbdc239e43ee272e3f817e8c1769c014}$ & Here & Conjectured \\
		$373$ & $\mathtt{0xe7c0b7ab0162ceb91ec853c2f76850ae}$ & Here & Conjectured \\
		$379$ & $\mathtt{0x1f2af197a8f25e0296596d20248591ff}$ & Here & -------- \\
		$383$ & $\mathtt{0xc97ce8db8b54026d78fe7f2c6509308aeff}$ & Here & Strong \\
		$389$ & $\mathtt{0x13bfdd4deb7ff4b44a05d1253e3e7fff}$ & Here & -------- \\
		$397$ & $\mathtt{0xe41c10582cab77c54a1de9a26ab8534c1}$ & Here & Conjectured \\
		$401$ & $\mathtt{0x3421ed1ee61d103a2eda6549932588fff}$ & Here & -------- \\
		$409$ & $\mathtt{0x7fd0db54b53f7f53cc6324419327fffff}$ & Here & -------- \\
		$419$ & $\mathtt{0x2675ef794ae8122ffcbef736715c91505c}$ & Here & -------- \\
		$421$ & $\mathtt{0x3febfce5c32dacdb532ccda188ba23c1f1}$ & Here & -------- \\
		$431$ & $\mathtt{0x131767ca5fd6a1667d3700895c3343b0997}$ & Here & Conjectured \\
		$433$ & $\mathtt{0x11237d1292cb65eb8f153339e7158815130ff}$ & Here & Strong \\
		$439$ & $\mathtt{0xcb51241c35e76ddfa7bc01b790845fe070f}$ & Here & Conjectured \\
		$443$ & $\mathtt{0x17555d2396cbcdc2f04b1c95db7543c4ad8}$ & Here & --- \\
		$449$ & $\mathtt{0xa96bf0ad86765b20cb1585cc5c7fc3e1000}$ & Here & Conjectured \\
		$457$ & $\mathtt{0x47e7534ec08aa0ade5c68ffaebe32c75dd66c}$ & Here & Conjectured \\
		$461$ & $\mathtt{0x2748aecc4cc7cb5c90412725cc61188a9fe0}$ & Here & -------- \\
		$463$ & $\mathtt{0x15a054b564abe3e572c7916d4cce03fbc79f}$ & Here & -------- \\
		$467$ & $\mathtt{0x25a3047abde565873d7201a20205b585363cec}$ & Here & Strong \\
		$479$ & $\mathtt{0x6452414968e36fccad4592ce0c0ae611f730}$ & Here & -------- \\
		$487$ & $\mathtt{0x129085408a169c0b64f3c3738916f1fea5eff}$ & Here & -------- \\
		$491$ & $\mathtt{0x1163755ef8f6405ebd8afeeaa81c3173c5400}$ & Here & -------- \\
		$499$ & $\mathtt{0x1f56702b1129e2b9fc67b684638a2149c34a8324}$ & Here & Strong \\
		$503$ & $\mathtt{0x6c590c78869ad5dc0c06bf46dbda0257f4480}$ & Here & -------- \\
		$509$ & $\mathtt{0x7833d2561b088645c9375eee7c1fe5baf317400}$ & Here & Conjectured \\
		$521$ & $\mathtt{0x2151a2dabe7e258bb00ab866e1e84b03980000}$ & Here & -------- \\
		$523$ & $\mathtt{0x48b907f88c7a58ad008b7a54ebcbf5b65a57a7}$ & Here & -------- \\
		$541$ & $\mathtt{0x2293d33ed8af59723d5aa8f5f376ca2a5d64d70}$ & Here & -------- \\
		$547$ & $\mathtt{0x9877de1375932504d620d81efe2b3dec7f7b6bf}$ & Here & -------- 
\end{tabular}} 
\caption{Largest strictly $B$-smooth twins for $127 \leq B \leq 547$, where they were found and our confidence in their optimality among the set of strictly $B$-smooth twins. We refer our optimality confidence as \emph{strong} if it is above the estimate and simply \emph{conjectured} if it lies on the estimate.} \label{tab:listlargestwins}

\end{table}


\fi